\documentclass[12pt,reqno]{amsart}
\usepackage[utf8]{inputenc}
\usepackage[vmargin=2.5cm,hmargin=3cm]{geometry} 
\usepackage[british]{babel}

\title{Conflict-avoiding codes and near-primitive roots}
\author[Karthick Babu]{C. G. Karthick Babu}
\address{School of Mathematics and Computer Science, Indian Institute of Technology Goa, Farmagudi, Ponda-403401, Goa, India}
\email{karthick@iitgoa.ac.in}
\author[Moree]{Pieter Moree}
\address{Max Planck Institute for Mathematics, Vivatsgasse 7, 53111 Bonn, Germany}
\email{moree@mpim-bonn.mpg.de}
\author[Pasupulati]{Sunil Kumar Pasupulati}
\address{Department of Mathematics, Sri A.S.N.M. Govt. College (A), Palakollu, West Godavari Dist.- 534260, Andhra Pradesh, India.}
\email{sunilkumarpasupulati@gmail.com}
\author[Sgobba]{Pietro Sgobba}
\address{Division of Natural and Applied Sciences, Duke Kunshan University, 8 Duke Avenue, Kunshan, China 215316}
\email{pietro.sgobba@dukekunshan.edu.cn}
\date{}

\usepackage{amssymb,latexsym}
\usepackage{comment}
\usepackage{titletoc}
\usepackage{amsfonts, amsmath, amsthm, fancybox, graphicx, color, times, babel}
\usepackage{enumitem}

\usepackage{multirow}

\usepackage[usenames,dvipsnames]{xcolor}

\usepackage[all,cmtip]{xy}
\newtheorem {thm}{Theorem}
\newtheorem* {thm*}{Theorem}

\newtheorem {cor}{Corollary}

\newtheorem* {cor*}{Corollary}
\newtheorem {lem}{Lemma}

\newtheorem {prop}{Proposition}
\newtheorem {defi}{Definition}

\theoremstyle{definition}
\newtheorem{rem}{Remark}

\newtheorem* {conj*}{Conjecture}

\newtheorem* {quest*}{Question}

\newcommand{\mc}{\mathcal}
\newcommand{\p}{\mathfrak{p}}
\newcommand{\Q}{\mathbb{Q}}
\newcommand{\F}{\mathbb{F}}
\newcommand{\Z}{\mathbb{Z}}

\DeclareMathOperator{\ind}{ind}
\DeclareMathOperator{\ord}{ord}
\DeclareMathOperator{\Li}{Li}

\DeclareMathOperator{\lcm}{lcm}

\definecolor{carmine}{rgb}{0.7, 0.1, 0.09}

\usepackage{mathtools}

\DeclarePairedDelimiterX\kronecker[2]{\lparen}{\rparen}{\frac{#1}{#2}}

\newcommand\legendre[2]{\kronecker{#1}{#2}}

\newcommand\biglegendre[2]{\kronecker*{#1}{#2}}

\usepackage[pdftex, draft=false, backref = page]{hyperref}
\hypersetup{
    colorlinks = true,
    linkcolor = blue,
    urlcolor = blue, 
    citecolor = blue}

\renewcommand*{\backrefalt}[4]{%
\ifcase #1 %
\color{red} No citations.%
\or
(p.~#2).%
\else
(pp.~#2).%
\fi
}



\usepackage{microtype}

\makeatletter
\patchcmd{\env@cases}{1.2}{1}{}{}
\makeatother

\usepackage{mathrsfs}

\begin{document}

\subjclass{Primary: 94B25; Secondary: 11N37}
\keywords{Conflict-avoiding codes, near-primitive root, multiplicative order,
asymptotic counting, graph, irreducibility, Cauchy-Frobenius identity}

\maketitle

\begin{abstract}
An important role in the theory of
conflict-avoiding codes is played
by a quantity involving the multiplicative order of $4$ modulo $d$, with $d$ running over all positive divisors of an odd integer $n$. 
We establish results on its distribution  as $n$ varies over
the odd prime numbers, respectively odd integers, using techniques from the study of Artin's primitive root conjecture. 
We also point out some other interpretations of this quantity, 
with one of them making a connection with the factorization of $X^n-1$ into irreducibles over the finite field $\mathbb F_2$.
\end{abstract}
\section{Introduction}
Conflict-avoiding codes have been investigated as protocol sequences for use in a multiple-access (collision) channel without feedback \cite{GV93, Mat90, NGM92, Lev07, TR02}. For a detailed technical description of this multiple-access channel model, we refer the reader to \cite{BP83, MM85}.

In mathematical terms, the definition is as follows: Given any positive integer $n$, we let $\Z_n$  denote the ring of integers modulo $n$, and let $\mathcal{P}(n, k)$ represent the set of all $k$-element subsets of $\Z_n$.
Each element $x \in \mathcal{P}(n, k)$ can be associated with a binary sequence of length $n$ and weight $k$, indicating the positions of the nonzero entries.
For a given subset $x \in \mathcal{P}(n, k)$, we define its difference set as follows: 
\[
\Delta(x)\coloneq \{a-b \pmod{n}: a,b \in x, \ a\neq b\}.
\]

Note that $|\Delta(x)| \leq k(k - 1)$,
and if $i \in \Delta(x)$, then $n - i \in \Delta(x)$ as well; that is, $\Delta(x)$ is symmetric with respect to $n/2$. A \emph{conflict-avoiding code} ($\mathrm{CAC}$) of length $n$ and weight $k$ is defined as a subset $\mathcal{C} \subset \mathcal{P}(n, k)$ such that for any distinct $x, y \in \mathcal{C}$, the intersection $\Delta(x) \cap \Delta(y)$ is empty. Each element $x \in \mathcal{C}$ is referred to as a \emph{codeword} of length $n$ and weight $k$. Without loss of generality, we can assume that all codewords contain $0$. For example $\mathcal{C}=\{\{0,1,6\}, \{0,4,8\}\}$ is a $\mathrm{CAC}$ of length 11 and weight 3, 
with the codewords corresponding to 
the binary protocol sequences 
$(1,1,0,0,0,0,1,0,0,0,0)$, respectively $(1,0,0,0,1,0,0,0,1,0,0)$.

For given $n$ and $k$, let $\mathrm{CAC}(n,k)$ denote the class of all $\mathrm{CAC}$s of length $n$ with weight $k$. The maximum possible 
size of the codes in $\mathrm{CAC}(n,k)$ will be denoted by $M(n,k)$, i.e.,
\[
M(n,k)\coloneq \max\{|\mathcal{C}| : \mathcal{C} \in \mathrm{CAC}(n,k)\}.
\]
A code $\mathcal{C} \in \mathrm{CAC}(n,k)$ is said to be \emph{optimal}, if $|\mathcal{C}|=M(n,k)$. The above example is optimal in $\mathrm{CAC}(11,3)$, that is, $M(11,3)=2$ (as is easily deduced using the Pigeonhole Principle). 
Zhang, Lo and Wong~\cite{ZLW} 
constructed several families of optimal strongly conflict-avoiding codes of even length and weight three.

\subsection{Equi-difference conflict-avoiding codes}
 A codeword $x \in P(n, k)$ is said to be \emph{equi-difference} (or \emph{centered} in the case $k = 3$) with generator $i \in \Z_n \setminus \{0\}$ if it has the form
\[
x = x_i = \{0, i, 2i, \dotsc, (k - 1)i\} \pmod{n},
\]
and consists of $k$ distinct residues modulo $n$.
A code $\mathcal{C}\subset P(n, k)$ is called an \emph{equi-difference conflict-avoiding code} if every codeword in $\mathcal{C}$ is equi-difference. It is  denoted $ \mathrm{CAC^e}(n, k) $. Note that for an equi-difference codeword \( x_i \), we have
\[
\Delta(x_i) = \{ \pm ji \pmod{n} :  1 \leq j \leq k-1 \},
\]
which implies that \( |\Delta(x_i)| \leq 2(k - 1) \). A codeword satisfying \( |\Delta(x_i)| \leq 2(k - 1) \) is said to be \emph{exceptional}. Moreover, an optimal code \( \mathcal{C} \in \mathrm{CAC}(n, k) \) is said to 
be \emph{tight} if 
\[
\cup_{x \in \mathcal{C}} \Delta(x)= \Z_{n} \setminus\{0\}.
\]
The 
maximum code size 
in $\mathrm{CAC}^{e}(n,k)$ will be denoted by $M^{e}(n,k)$, i.e.,
\[
M^{e}(n,k)=\max\{|\mathcal{C}| : \mathcal{C} \in \mathrm{CAC}^{e}(n,k)\}.
\]

For the remainder of this section we take $k=3$.
It is easy to verify that a general codeword \( x \) satisfies
\[
2 \leq |\Delta(x)| \leq 6.
\]
Moreover, if \( |\Delta(x)| = 4 \), then \( x \) must be an equi-difference codeword. 
For exceptional codewords, we have:
\begin{itemize}
    \item $|\Delta(x)| = 2$ if and only if $3 \mid n$ and  $x = \{0, n/3, 2n/3\}$,
    \item $|\Delta(x)| = 3$ if and only if $4 \mid n$ and $x = \{0, n/4, n/2\}$,
\end{itemize}
and both of these are also equi-difference. 

\subsection{A graph associated to  \texorpdfstring{$\mathrm{CAC}(n, 3)$}{CACn3}}
 For simplicity of notation, we write $\mathrm{CAC}(n)$ and $M(n)$ instead of $\mathrm{CAC}(n, 3)$, 
respectively $M(n, 3)$. Likewise, we denote $\mathrm{CAC}^{e}(n, 3)$ and $M^{e}(n, 3)$ by $\mathrm{CAC}^{e}(n)$, respectively $M^{e}(n)$. 
Following  Momihara \cite{Momi} we now define a graph $G_n$ connected with equi-difference codewords. A closely related graph-theoretic description is also used by Mishima and Momihara \cite{miMo} in
their construction of a new family of optimal tight conflict-avoiding codes.
\begin{defi}
Given an odd natural number $n$, define a graph $G_n = (V, E)$ as follows:
	\begin{itemize}
		\item The vertex set is 
		\[
		V = \{1, 2, \dotsc, (n-1)/2\}.
		\]
		\item The edge set $E \subseteq V \times V$ consists of ordered pairs $(a,b)$ such that
		\[
		b \equiv \pm 2a \pmod{n}.
		\]
	\end{itemize}
\end{defi}
A vertex $i \in V(G_n)$ can be viewed as representing the difference $\pm i$, and thus, an edge $(a, b)$ corresponds to the equi-difference codeword ${0, a, 2a}$. A loop in $G_n$ represents the special equi-difference codeword ${0, n/3, 2n/3}$, which arises only when $3 \mid n$. The graph $G_n$ decomposes as a collection of disjoint cycles. This follows from the fact that the edge condition defines a permutation on the vertex set $V$. By identifying $x$ and $-x$, the set $V$ represents the non-zero elements of $\Z_n$ up to sign. The edge condition $b \equiv \pm 2a \pmod{n}$ acts as multiplication by 2 in this quotient space. Since by assumption $n$ is odd, the number $2$ is invertible modulo $n$. Consequently, the mapping $a \mapsto \pm 2a \pmod{n}$ is a bijection, guaranteeing that every vertex in $G_n$ has an in-degree and out-degree of exactly 1. 
Equivalently, the reachability between vertices can be formalized by the relation $a \sim b$ if and only if $b \equiv \pm 2^r a \pmod{n}$ for some integer $r$. Because 2 is invertible in $\Z_n$, this satisfies the properties of an equivalence relation. The resulting equivalence classes partition $V$ into distinct connected components, which strictly correspond to the disjoint cycles of the graph.

\begin{defi}\label{defn of On}
Given any odd natural number $n$, we let $\mc{O}(n)$ denote the number of  cycles of odd length in $G_n$, where a loop is considered a cycle of length 1.
For convenience, we put $\mc{O}(1) = 0$.
\end{defi}
\noindent For example, $G_{21}$ has three cycles: $(7)$, $(3, 6, 9)$, and $(1, 2, 4, 8, 5, 10)$, 
and hence $\mc{O}(21) = 2$.  The cycle $(7)$ is considered a loop.

Recall that for any two distinct codewords $x \neq y$, the condition $\Delta(x) \cap \Delta(y) = \varnothing$ must be satisfied. To construct as many codewords as possible, we aim to maximize the number of disjoint vertex subsets, where each vertex subset $S$ falls into one of the following three categories:
\begin{enumerate}[label=(\alph*)]
    \item $|S| = 1$: the induced subgraph is a loop, \label{loop}
    \item $|S| = 2$: the induced subgraph is a single edge, \label{equi-diff}
    \item $S = \{a, b, c\}$ such that $a + b \equiv \pm c \pmod{n}$. \label{non equi-diff}
\end{enumerate}
Note that for any odd integer  $n$, the graph $G_n$ contains at most one loop. Since each equi-difference codeword corresponds to an edge in $G_n$, constructing an optimal code $\mathcal{C} \in \mathrm{CAC}^{e}(n)$ amounts to finding a maximal matching in  $G_n$, excluding the loop (recall that a \emph{matching} $M$ in 
a graph $G=(V,E)$ is a subset of the edge set $E$ such that no two distinct edges in $M$ share a common vertex). Moreover, since each vertex $i \in V(G_n)$ represents the difference $\pm i$, a tight equi-difference code $\mathcal{C} \in \mathrm{CAC}^{e}(n)$ exists if and only if there is a matching that covers all vertices of $G_n$. Given that $G_n$ is a disjoint union of cycles, such a matching is possible if and only if every cycle has even length. As a consequence, we obtain the following proposition.
\begin{prop}\cite[Prop.\,1]{FLS} \label{tight equi-diff prop}
Given an odd integer $n \geq 3$, a tight equi-difference code $\mathcal{C} \in \mathrm{CAC}^{e}(n)$ exists if and only if $3\nmid n$ and  $\mc{O}(n) = 0$, or $3\mid n$ and $\mc{O}(n) = 1$.
\end{prop}
The reader might wonder why the
$3$ divisibility of $n$ plays a 
role. This is due to the existence of the special equi-difference codeword $(0,n/3,2n/3)$ which of course
only occurs if $3\mid n$.

Making use of some existing results it is not difficult to quantify Proposition 
\ref{tight equi-diff prop} (for a proof see Section \ref{tightequicodeproof}).  The big oh notation for the error terms is explained in any book on analytic number theory, for example in that of Apostol \cite[p.\,53]{Apostol}. The notation $\mc{O}$ is in this paper solely used for the quantity defined in Definition \ref{defn of On}.
\begin{thm}\label{tight equi code thm}
Let $P(x)$ and $N(x)$ denote the number of primes, respectively odd integers $1\le n\le x$, for which a tight equi-difference code of length $n$ exists.
We have
\[
P(x)=\frac{5}{12}\frac{x}{\log x}+O\Bigl(\frac{x\log\log x}{(\log x)^{7/6}}\Bigr)\quad\text{~and~}\quad
N(x)=C\frac{x}{(\log x)^{7/12}}+O\Bigl(\frac{x(\log\log x)^5}{(\log x)^{19/12}}\Bigr),
\]
with $C$ a positive constant. 
More precisely, if $C_{0}$ denotes the constant appearing in Theorem \ref{density O(m) is zero thm}, then $C=\frac{4}{3}C_{0}\approx 0.4298\dotsc$
\end{thm}
More generally, Proposition \ref{tight equi-diff prop} makes clear the importance of understanding which values are assumed by $\mc{O}(n)$ and for how many odd 
integers  $1\le n\le x$ they get assumed as $x$ gets large. This is a delicate question that we intend to discuss in greater depth in a future paper.
Here we present some alternative interpretations of 
$\mc{O}(n)$ and study $\mc{O}(p)$ in detail as $p$ runs through the primes,
something that also has a coding theoretic relevance, as becomes particularly
clear in Section \ref{prime-values sec}. The techniques used will come mainly from the analytic and algebraic number theory used to deal with variations of Artin's primitive root conjecture. A good inkling of the mathematics involved can be obtained from the survey \cite{Msurvey}. Many results in this area, also in this paper, will be under the \emph{Generalized Riemann Hypothesis} (GRH), by which we mean that each Dedekind zeta-function has only non-trivial zeros on the half-line $\Re(s)=1/2$. Indeed, typically this assumption is only required for the Dedekind zeta-functions of the number fields involved in the proof of the result.
\subsection{Description of the  main results}
Before turning to the distributional questions, we first develop several complementary descriptions of the quantity $\mc{O}(n)$. We recall the formula of Fu, Lo and Shum~\cite{FLS} in terms of $\ord_d(2)$ and derive an equivalent formulation using $\ord_d(4)$. In particular, for an odd prime $p$,  this gives a particularly simple description of $\mc{O}(p)$ in terms of the index of $4$ modulo $p$. 
This description also admits a group-theoretic interpretation. Introducing the graph $H_p$, whose vertices are the non-zero squares modulo $p$ and whose edges are induced by multiplication by $4$, we show that $H_p$ is isomorphic to $G_p$. Consequently, whenever $\mc{O}(p)>0$,
\[
\mc{O}(p)=[(\F_p^\times)^2:\langle 4\rangle]=[\F_p^\times:\langle -1,2\rangle].
\]
This also gives a short reproof of a result of Fu, Lo and Shum~\cite{FLS}.

A different interpretation of $\mc{O}(n)$ is obtained through the factorization of $x^n-1$ over $\F_2$. More precisely, $\mc{O}(n)$ can be recovered by counting certain pairs of distinct reciprocal irreducible factors of odd degree together with self-reciprocal irreducible factors whose degrees are congruent to $2$ modulo $4$. This interpretation recovers the preceding order formulas and also leads to the characterization of Wu and Fu~\cite{WuFu} in case the length of the code is a number of the form $2^e\pm1$.

Finally, by viewing multiplication by $a$ as an automorphism of the additive group $\mathbb Z_m$, we use the Cauchy--Frobenius identity to obtain an explicit formula for the number of odd-length orbits of a general multiplication map. Specializing this description to multiplication by $2$ yields another formula for $\mc{O}(n)$.

In the remaining part of the section we describe our analytic number theory results.
\subsubsection{Results on the distribution of $\mc{O}(p)$ for prime $p$}
Let $k \ge 0$ be an integer.
Put $$P_k\coloneq \{p>2:{\mc{O}}(p)=k\}$$ and let $P_k(x)$ denote the number of primes $p\leq x$ in $P_k$. The counting function $P_0(x)$ satisfies, by Proposition \ref{tight equi-diff prop},  the same asymptotic as that given for $P(x)$ in Theorem \ref{tight equi code thm}.
\begin{thm}[under GRH]
\label{main}
Let $k \ge 1$ be an integer.
The set of primes $P_k$ has a natural density $\delta_k$ given by
$$
\delta_k =
A\cdot E(k) \cdot
\begin{cases}
1 & \ \textrm{if}\ k{\text~is~odd};\\
0 & \ \textrm{if}\ k\equiv 2\pmod{4};\\
2/3 & \ \textrm{if}\ 4\mid k,
\end{cases}
$$
with
\begin{equation*}
\label{Artinconstantdef}
A\coloneq \prod_{\textrm{prime $q$}}\Big (1-\frac{1}{q(q-1)} \Big) = 0.3739558136192022880547280543464\dotsc,
\end{equation*} 
the Artin constant and
\begin{equation}
\label{def:Ek}    
E(k)\coloneq \frac{1}{k^2}\prod_{q\mid k}\frac{q^2-1}{q^2-q-1}.
\end{equation}
Moreover, asymptotically in $x$,
\[  P_k(x) = \delta_k \frac{x}{\log x} + O_{k}\Big(\frac{x\log\log x}{\log^2x} \Big). \]
In case $k\equiv 2\pmod{4}$, we have $P_k(x)=0$.
\end{thm}

The techniques used to deal with $P_k$ can also be used to study the level sets of $\mc{O}(p)$ within a prescribed arithmetic progression. For coprime positive integers $a$ and $d$, put
$$
P_k(a,d)(x)\coloneq\#\{p\leq x:p\equiv a\pmod d,\ \mc O(p)=k\}.
$$
For $k\geq1$, Theorem~\ref{thmgeneral} gives, under GRH, an asymptotic formula for $P_k(a,d)(x)$. The cases $k=0$ and $k=1$ admit particularly explicit descriptions. 
For $k=0$ the result is unconditional and the corresponding rational density is described explicitly in Theorem~\ref{dens zero prop}.
For $k=1$, we obtain the following more explicit result.
\begin{thm}\label{thm:p1ad}
Let $a$ and $d$ be coprime positive integers. 
Under  GRH, we have
\begin{equation}\label{eq:underGRH}
P_1(a,d)(x)= 
\frac{\rho(d,a)}{\varphi(d)}\frac{x}{\log x}
+
O_d\Bigl(\frac{x\log\log x}{(\log x)^{2}}\Bigr),
\end{equation}
with
$$
\rho(d,a)=c(d,a)R(d,a)A,
$$
where
\[
R(d,a)\coloneq 2\prod_{p\mid (a-1,d)}\Big(1-\frac{1}{p}\Big)
\prod_{p\mid d}\Big(1+\frac{1}{p^2-p-1}\Big),
\]
and
$$
c(d,a)\coloneq 
\begin{cases}1/2  & \textrm{if~}4\nmid d;\\
1  & \textrm{if~}4\mid d,~a\equiv 3 \pmod{4};\\
0  & \textrm{if~}4\mid d,~a\equiv 1 \pmod{4}.
\end{cases}
$$
\end{thm}

\medskip
Our distribution result for $\mc{O}(p)$ also has a direct application to the determination of the optimal size of conflict-avoiding codes of prime length. If $\mc{O}(p)>0$, then
\begin{equation}
\label{M(p) ineq} 
M^{e}(p)
\leq M(p)
\leq M^{e}(p)+\Bigl\lfloor\frac{\mc{O}(p)}{3}\Bigr\rfloor,
\end{equation}
as will be explained in Section~\ref{Alter gp for O(p) sec}.
A result of Hsia, Li and Sun~\cite{HLScyclo} shows that the upper bound is attained when the index $\mc{O}(p)$ is an odd prime. Combining their result with Theorem~\ref{main} gives the following consequence.
\begin{thm}\label{CAC-prime-values thm}
Let $p$ be an odd prime such that $\mc{O}(p)$ is an odd prime. Then the upper bound in \eqref{M(p) ineq} is attained, and hence
\[
M(p)=M^{e}(p)+\Bigl\lfloor\frac{\mc{O}(p)}{3}\Bigr\rfloor=\frac{p-1}{4}
-\frac{\mc{O}(p)}{2}+\Bigl\lfloor\frac{\mc{O}(p)}{3}\Bigr\rfloor.
\]
Moreover, assuming GRH, the set of such primes has relative density $0.508822711238\dotsc$ within the set of odd primes $p$ satisfying $\mc{O}(p)>2$.
\end{thm}
\subsubsection{Result on the distribution of $\mc{O}(n)$ for integer $n$}
Given any integer $e\geq0$, one can wonder if there are $n$ with  $\mc{O}(n)=e$ and if they exist, how frequent they are. The following result gives a partial answer (for a proof see Section \ref{Onintegerdistribution}).

\begin{thm}\label{density O(m) is zero thm}
For $x\geq 3$ and $e\ge 0$, let 
\[
N_{e}(x) \coloneq \#\{1 < m \leq x:m \ \textrm{odd and} \ \mc{O}(m)=e\}.
\]
We have
\[
N_{0}(x)=C_{0}\frac{x}{(\log x)^{7/12}}
+O\Bigl(
\frac{x(\log\log x)^5}{(\log x)^{19/12}}
\Bigr),
\]
with $$C_0=\frac{1}{\Gamma(\frac{5}{12})}\lim_{t\rightarrow \infty}\prod_{p\le t}
g(p)\Bigl(1-\frac{1}{p}\Bigr)^{\frac{5}{12}}=0.32235\dotsc,$$
and $g(p)=(1-1/p)^{-1}$ if $2\mid \ord_p(4)$ and $g(p)=1$ otherwise.
For every fixed $e\ge 1$ and for sufficiently large $x$, we have  
$$N_e(x)
\gg_{e} \frac{x}{(\log x)^{7/12}}.$$
\end{thm}
Using Artin primitive root techniques one can say rather more in case $e\ge 1$.
We hope to come back to this issue in a future publication.

\subsection{Organization of the paper}
Sections~\ref{sec:formulas-for-On}--\ref{Alter interp sec} develop order-theoretic, graph-theoretic, polynomial, and orbit-counting descriptions of $\mc{O}(n)$, recovering several earlier results along the way. Section~\ref{initial consid sec} introduces the near-primitive-root framework and some preliminary density computations. Section~\ref{Pf of thm 2 sec} applies these tools to prove Theorem~\ref{main}. Section~\ref{Mult ord sec} studies the distribution of $\mc{O}(p)$ in arithmetic progressions and proves Theorem~\ref{thm:p1ad}. Section~\ref{prime-values sec} considers prime values of $\mc{O}(p)$ and proves Theorem~\ref{CAC-prime-values thm}, while Section~\ref{Onintegerdistribution} treats the distribution of $\mc{O}(n)$ over odd integers and proves Theorem~\ref{density O(m) is zero thm}. Finally, Section~\ref{pf of thm 1 sec} proves Theorem~\ref{tight equi code thm}.

\section{Formulas for \texorpdfstring{$\mc{O}(n)$}{O(n)}}\label{sec:formulas-for-On}
We begin by recalling a formula of Fu, Lo and Shum for the number of odd
cycles in $G_n$.  We then give an equivalent formulation (Proposition \ref{prop:ord4-formula}), involving only
the multiplicative order of $4$, which will be particularly
convenient in the subsequent sections.
\begin{defi}[order and index] \label{mult ord def}
Let $a\in \Q^{\times}$ and $n\ge 2$. Assume that $\nu_p(a)=0$ for every prime $p\mid n$; equivalently, if $a=r/s$ is written in lowest terms, then $\gcd(rs,n)=1$. The reduction of $a$ modulo $n$ is then the element $
\overline{a}:=r s^{-1}\in (\Z_{n})^{\times}$. Then the smallest integer $k\ge 1$ such that $\overline{a}^k\equiv 1\pmod{n}$ is called the \emph{multiplicative order} of $a$ modulo $n$ and denoted by $\ord_n(a)$. Differently phrased it is the order of the subgroup $\langle \overline{a}\rangle$ generated by $\overline{a}$ in $\Z_{n}^{\times}$. The index of that subgroup is denoted by $\ind_n(a)$ and called the \emph{index} of $a$ in $\Z_{n}^{\times}$. We have
$$\ind_n(a)=\frac{\varphi(n)}{\ord_n(a)}.$$
\end{defi}
If space allows it, we prefer to write out the expression for the index in terms of the order.
\begin{thm}[Theorem 1 of \cite{FLS}]
\label{thm-FLS}
For odd $n \geq 3$, we have 
$$\mc{O}(n)=\sum_{d\, \mid \, n} \delta_{d},$$
where
$$
\delta_{d}=
\begin{cases}
\frac{1}{2}\ind_d(2) & \ \textrm{if}\ \ord_d(2)\textrm{~is~odd~and~} >1;\\[1ex]
\ind_d(2) & \ \textrm{if}\ \ord_d(2) \equiv 2 \pmod{4}, \ 2^{\ord_d(2)/2}\equiv-1\pmod{d};\\[1ex]
0 & \textrm{~otherwise}.
\end{cases}
$$
\end{thm}
The next proposition is a variant of Theorem~\ref{thm-FLS} involving only $\ord_d(4)$, rather than $\ord_d(2)$.
\begin{prop}\label{prop:ord4-formula}
For every odd integer $n\geq 3$, we have
\begin{equation}\label{O(n) defn interms ord of 4}
    \mc{O}(n)=\sum_{\substack{d\, \mid \, n, \ d >1 \\ 2\,\nmid\, \ord_{d}(4)\\ 2^{\ord_d(4)}\equiv\pm1\pmod{d}}} \frac{\varphi(d)}{2\ord_{d}(4)}.
\end{equation}
\end{prop}
\begin{proof}
Identify the vertices of $G_n$ with the nonzero residue classes modulo
$n$, up to sign.  For each divisor $d>1$ of $n$, let
\[
  \Omega_n(d)\coloneq \left\{[a]\in (\Z _n\setminus\{0\})/\{\pm1\}:\frac{n}{\gcd(a,n)}=d\right\}.
\]
The set $\Omega_n(d)$ has $\varphi(d)/2$ elements and is invariant under the map $[a]\mapsto[2a]$.  Every cycle contained in $\Omega_n(d)$ has length $s_d\coloneq \min\{r\geq1:2^r\equiv\pm1\pmod{d}\}$. Put $q_d=\ord_d(4)$. Since $4^{q_d}\equiv1\pmod{d}$, we have $2^{2q_d}\equiv1\pmod{d}$, and hence $s_d\mid 2q_d$. Conversely, if $2^r\equiv\pm1\pmod{d}$, then $4^r\equiv1\pmod{d}$;
therefore $q_d\mid r$.  It follows that
\[
  s_d=
  \begin{cases}
     q_d,  & \text{if }2^{q_d}\equiv\pm1\pmod{d},\\
     2q_d, & \text{otherwise}.
  \end{cases}
\]
Consequently, the cycles contained in $\Omega_n(d)$ have odd length if
and only if $q_d$ is odd and $2^{q_d}\equiv\pm1\pmod{d}$. When these conditions hold, $\Omega_n(d)$ consists of
\[
  \frac{|\Omega_n(d)|}{q_d}=\frac{\varphi(d)}{2q_d}
\]
odd cycles.  Summing over the divisors $d>1$ of $n$, we obtain \eqref{O(n) defn interms ord of 4}.
\end{proof}
\begin{rem}
\label{firstremark}
The congruence condition in \eqref{O(n) defn interms ord of 4} is not redundant
for composite $d$, since the congruence $x^2\equiv1\pmod{d}$ may have
solutions other than $x\equiv\pm1\pmod{d}$.  When $d=p$ is an odd prime,
however, the condition is automatic.  Thus
\[
  \mc{O}(p)=
  \begin{cases}
  \frac{1}{2}\ind_p(4),
      & \text{if $\ord_p(4)$ is odd};\\[8pt]
  0,  & \text{otherwise}.
  \end{cases}
\]
The graph-theoretic meaning of this prime-length specialization will be
developed in the next section.
\end{rem}

\section{A new interpretation of \texorpdfstring{$\mc{O}(p)$}{Op} through a graph equivalent of \texorpdfstring{$G_p$}{Gp}} \label{Alter gp for O(p) sec}
We introduce a graph that is simpler to analyze yet structurally equivalent to $ G_n $; that is, it is isomorphic to $ G_n $ in some cases.
Subsequently, we show how it can be used to give a new interpretation of $\mc{O}(p)$ with $p\ge 3$ a prime.
\begin{defi} Let $n$ be a natural number. Define a graph $H_n = (V, E)$ as follows:
\begin{itemize}
	\item The vertex set $V$ is the set of non-zero squares in $\Z_{n}$, i.e.,
	\[
	V = \{ x^2 \pmod{n} : x \in \Z_n \} \setminus \{0\}.
	\]
	\item The edge set $E \subseteq V \times V$ consists of ordered pairs $(a,b)$ such that
	\[
	b \equiv 4a \pmod{n}.
	\]
\end{itemize}
\end{defi}

\begin{prop}
    Let $p\geq 3$ be a prime.
Then $G_p$ and $H_p$ are isomorphic as directed graphs.
\end{prop}

\begin{proof}
    
Define the map 
\[
\phi  \colon  V(G_p) \to V(H_p) \quad \textrm{by} \quad \phi(a) = a^2 \pmod{p}.
\]
Since $p$ is prime, it implies that there are only two elements  $a, p-a \in \Z_{p} \setminus \{0\}$  whose square is equal to $a^2$.
Therefore, the map $\phi  \colon  V(G_p) \to V(H_p) $ is injective and surjective.
Now, suppose that there is a directed edge $a \to b$ in $G_p$. Then $b \equiv \pm 2a \pmod{p}$. Consider
\[
\phi(b) = b^2 \equiv (\pm 2a)^2 = 4a^2 \equiv 4 \phi(a) \pmod{p},
\]
so that there is an edge from $\phi(a)$ to $\phi(b)$ in $H_p$.
Hence, $\phi$ is a homomorphism of graphs.

In order to see that $\phi$ is a graph isomorphism, we only need to prove that whenever there is an edge between  $\phi(a) $  and $\phi(b)$, there is an edge between $a$ and $b$. 
 If $(\phi(a),\phi(b)) \in E(H_p)$,
 then $b^2-4a^2\equiv 0 \pmod{p}$ and 
 hence either $p\mid b+2a  $ or $p\mid b-2a $. Thus, there is an edge connecting  $a$ and $b$ in the graph $G_p$.
Hence, $\phi$ is a graph isomorphism from $G_p$ to $H_p$. \end{proof}

\begin{rem}
If $ n $ is not a prime number, then
$G_n$ is not necessarily isomorphic 
to $H_n$. Even in the case where $ n $ is squarefree, several challenges arise, namely:
    \begin{itemize}
        \item The sets of vertices in the corresponding graphs may not be identical in size.
        \item The  congruence $ x^2 \equiv 1 \pmod{n} $ can have up to $ 2^{\omega(n)} $ solutions, which may result in the presence of edges in $ H_n $ that do not appear in $ G_n$ (as usual $\omega(n)$ denotes the number of distinct prime divisors of $n$).
    \end{itemize}
   
\end{rem}
Although the graphs  $G_p$ and $H_p$ are isomorphic whenever $p$ is prime, the graph $H_p$ is often easier to analyze. This is because its vertex set  $V$ forms a subgroup of $\Z_{p}^{\times}$, and the structure of $H_p$ reveals that each disjoint cycle corresponds to a coset of the subgroup $\langle 4 \rangle$ in $V$. Since every cycle in $H_p$ has the same length, namely $\ord_p(4)$, 
the same holds for the graph $G_p$. Therefore, we have the following corollary:
\begin{cor}\label{no of odd cycle cor}
 \label{Opequivalence}
Let $p$ be an odd prime.
The number $\mc{O}(p)$ of odd cycles in $G_p$ is given by
$$
\mc{O}(p)=
\begin{cases}
[(\F_p^{\times})^{2} : \langle 4\rangle]=\frac{1}{2}\ind_p(4) & \ \textrm{if}\ \ord_p(4)\ \textrm{is odd};\\
0 & \textrm{~otherwise}.
\end{cases}
$$
\end{cor}

\begin{rem}
\label{rem:equalindices}
The squaring map on $\F_p^{\times}$ has kernel $\{\pm1\}$ and image $(\F_p^{\times})^2$. Its restriction to $\langle -1,2\rangle$ has the same kernel and image $\langle4\rangle$. Hence $|\F_p^{\times}|=2|(\F_p^{\times})^2|$, $|\langle-1,2\rangle|=2|\langle4\rangle|$,
and therefore
\[
[(\F_p^{\times})^2:\langle4\rangle]=[\F_p^{\times}:\langle-1,2\rangle].
\]
The index $[\F_p^{\times}:\langle-1,2\rangle]$
occurs naturally in constructions of optimal conflict-avoiding codes. Hsia, Li and Sun~\cite{HLSdiagonal} related the relevant coset condition to the solvability of certain diagonal equations over finite
fields. In this context they also pointed out the Artin-type problem of determining primes for which this index takes a prescribed value. Their results concern the construction of optimal codes once the index is
fixed, whereas our interest here is in the distribution of the corresponding values of $\mc{O}(p)$.
\end{rem}

\noindent On noting that 
\[
\ord_p(4)=
\ord_p(2) \cdot 
\begin{cases}
\frac12 & \ \textrm{if}\ \ord_p(2)\textrm{~is~even};\\[4pt]
1 & \textrm{~otherwise},
\end{cases}
\]
we can alternatively write 
$$\mc{O}(p)=
\ind_p(2) \cdot
\begin{cases}
\frac12 & \ \textrm{if}\ \ord_p(2)\ \textrm{is odd};\\[1ex]
1 & \ \textrm{if}\ \ord_p(2)\equiv 2\pmod{4};\\[1ex]
0 & \textrm{~otherwise}.
\end{cases}
$$
{}From this 
and the second supplementary law of quadratic reciprocity, stating that
\begin{equation}
\label{QR}    
\biglegendre{2}{p}=(-1)^{\frac{p^2-1}8}
=
\begin{cases}
+1 & \ \textrm{if}\ p\equiv \pm 1\pmod{8};\\
-1 & \ \textrm{if}\ p\equiv \pm 3\pmod{8},
\end{cases}
\end{equation}
where $\legendre{\cdot}{p}$ denotes the Legendre symbol,
we obtain an easy reproof of the following result of 
Fu, Lo and Shum \cite{FLS}.
\begin{thm}[Theorem 3 of \cite{FLS}]
If $p$ is an odd prime, then
$$
\mc{O}(p)=\ind_p(2) \cdot
\begin{cases}
\frac{1}{2} & \textrm{if}\  p \equiv 7 \pmod{8}, \ \textrm{or}\ \ p \equiv 1 \pmod{8} \ \textrm{and}\  \ord_p(2)\ \textrm{is odd};\\[1ex]
1 & \textrm{if}\  p \equiv 3 \pmod{8}, \ \textrm{or}\ \ p \equiv 1 \pmod{8} \ \textrm{and}\  \ord_p(2)\equiv 2\pmod{4};\\[1ex]
0 & \textrm{if}\  p \equiv 5 \pmod{8}, \ \textrm{or}\ \ p \equiv 1 \pmod{8} \ \textrm{and}\  4 \mid \ord_p(2).
\end{cases}
$$
\end{thm}

Apart from its role in the existence of tight equi-difference codes, the quantity $\mc{O}(n)$ is also relevant to the computation of $M(n)$ and $M^{e}(n)$. When studying conflict-avoiding codes of prime length $p$ and weight $3$, it is evident that there are no exceptional codewords. For an odd prime $p > 3$, the difference set $\Delta(x)$ has size $4$ if $x$ is an equi-difference codeword, and $6$ otherwise. It follows from Corollary~\ref{no of odd cycle cor} that the graph $G_p$ consists entirely of either odd or even cycles, depending on whether $\ord_{p}(4)$ is odd or even, respectively.

To construct an optimal $\mathrm{CAC}$, it is natural to include as many equi-difference codewords as possible. By Proposition~\ref{tight equi-diff prop}, a tight equi-difference code of length $p$ exists if and only if $\ord_{p}(4)$ is even. In this case, by the definition of tightness, we have $M(p) = M^{e}(p) = \tfrac{p-1}{4}$.

When $\ord_{p}(4)$ is odd, we can still construct an equi-difference code $\mathcal{C} \in \mathrm{CAC}^{e}(p)$ that includes all but one vertex from each cycle in $G_p$. Equivalently, there exists a matching in $G_p$ that covers all but one vertex of every cycle. This leads us to the following result:
\[
M^{e}(p)=\frac{p-1}{4}-\frac{\mc{O}(p)}{2}.
\]
As a consequence, we have $M(p) = M^{e}(p)$ whenever $\ord_{p}(4)$ is even. Moreover, by case \ref{non equi-diff} in Section~1.2 any non-equi-difference codeword $x=\{0, a, c\}$ corresponds to a subset $S = \{a, b, c\}$ satisfying the condition $a + b \equiv \pm c \pmod{p}$. Therefore, if $\ord_{p}(4)$ is odd, we obtain the upper and lower bound
for  $M(p)$ given in \eqref{M(p) ineq}.
Furthermore, the authors of \cite{FLS} propose an algorithm for constructing non-equi-difference codewords. They conjecture that, for any prime length $p$, the algorithm successfully produces $\bigl\lfloor \tfrac{\mc{O}(p)}{3} \bigr\rfloor$ such codewords. Subsequent work of Hsia, Li and Sun
\cite{HLSdiagonal,HLScyclo} showed that the upper bound in \eqref{M(p) ineq} is attained in several cases. This connection between $\mc{O}(p)$ and the upper bound in~\eqref{M(p) ineq}
will be combined in Section~\ref{prime-values sec} with our density results to prove
Theorem~\ref{CAC-prime-values thm}.

\begin{rem}
\label{rem:Girstmair}
The order condition occurring in Corollary \ref{no of odd cycle cor} also has an interpretation
in terms of digit expansions of rational numbers. Recall that the length
of the period of the base-$b$ expansion of $1/p$ is
$\ord_p(b)$. Recent work of Girstmair
\cite{G1,G2} gives formulas for the mean,
and more generally for the mean and variance, of the digits in such a
period under prescribed conditions on $\ord_p(b)$.
Specializing to $b=4$ gives a direct connection with $\mc{O}(p)$.

Indeed, if $\mc{O}(p)>0$, then Corollary \ref{no of odd cycle cor} gives
\[
  \ord_p(4)=\frac{p-1}{2\mc{O}(p)},
\]
and this order is odd. Girstmair's formula for an odd period of length $\ell=\frac{p-1}{q}$ therefore applies with $q=2\mc{O}(p)$. On the other hand, if $\mc{O}(p)=0$, then
$\ord_p(4)$ is even, and the digits in one period of the
base-$4$ expansion of $1/p$ have mean $3/2$. Since the correction term
in Girstmair's formula for an odd period is nonzero, combining his result
with Corollary \ref{no of odd cycle cor} gives the characterization
\[
  \mc{O}(p)=0
  \quad\Longleftrightarrow\quad
  \textrm{the digits in a period of the base-$4$ expansion of}\ \,
  \frac{1}{p}\ \,
  \textrm{have mean }\frac{3}{2}.
\]
More specifically, Girstmair \cite{G2} treats periods satisfying $\ord_p(b)=\frac{p-1}{2^m}$. For $b=4$, in the cases in which this order is odd, Corollary~1 gives $\mc{O}(p)=2^{m-1}$. Thus his mean and variance formulas provide arithmetic information about the base-$4$ digit expansions attached to these particular level sets of $\mc{O}(p)$.
\end{rem}

\section{An Alternative Interpretation of \texorpdfstring{$\mc{O}(n)$}{MC0n}}\label{Alter interp sec}
In this section, we present an alternative interpretation of $\mc{O}(n)$, which involves the concept of cyclotomic cosets and good integers.
\subsection{Cyclotomic coset, minimal polynomials}
In Section~\ref{sec:formulas-for-On}, we obtained two formulas for $\mc{O}(n)$ directly from the cycle structure of $G_n$.  We now recover these formulas from the factorization of $x^n-1$ over $\F_2$. This provides a different interpretation of $\mc{O}(n)$ in terms of reciprocal pairs and self-reciprocal irreducible factors.

Let $\F_{q}$ denote a finite field of $q=p^{m}$ elements where $p$ is a prime and $m$ is a positive integer. Let $n$ be a positive integer such that $(n,q)=1$. The cyclotomic coset of $q$ modulo $n$ containing $i$ is defined by
\[
C_{i}=\{i \cdot q^{j} \pmod{n} : j=0,1, \dots\}.
\]
A subset $\{i_{1},\dotsc,i_{t}\}$ of $\Z_{n}$ is called a complete set of representatives of cyclotomic cosets of $q$ modulo $n$ if $C_{i_1}, \dotsc, C_{i_{t}}$ are distinct and $\cup_{j=1}^{t}C_{i_{j}}= \Z_{n}$. 
The following is a well-known result which connects cyclotomic cosets with the factorization of $x^{n}-1$ into monic irreducible polynomials over $\F_{q}$.
\begin{prop}
\label{cosetfactorization}
Let $q$ be a prime power and $n$ be a positive 
integer coprime to $q$. Let $\alpha$ be a primitive $n$-th root of unity in $\F_{q^m}$ with $m=\ord_{n}(q)$. 
Let $\{s_{1},\dotsc,s_{t}\}$ be a complete set of representatives of cyclotomic cosets of $q$ modulo $n$. Then the polynomial $x^{n}-1$ factorizes into monic irreducible polynomials over $\F_{q}$ as follows:
\begin{equation}\label{irred fact of cyclo eq}
x^{n}-1= \prod_{i=1}^{t} \prod_{j \in C_{s_{i}}} (x-\alpha^{j})= \prod_{i=1}^{t}f_i(x).
\end{equation}
\end{prop}
From this result it is not difficult to derive the following classical result (Theorem 2.47 of \cite{LidlNiederreiter1, LidlNiederreiter2}).
\begin{lem}
\label{noofirrfactors}
Let $q$ be a prime power and $n$ be a positive 
integer coprime to $q$.
The polynomial $x^{n}-1$ factors
into
\begin{equation}\label{no of irred fact of cyclo}
 i_{q}(n)=\sum_{d\, \mid \, n}\frac{\varphi(d)}{\ord_{d}(q)}=\sum_{d\, \mid \, n}\ind_d(q)
\end{equation}
distinct irreducible factors over $\F_{q}$.
\end{lem}
\begin{proof}
By Proposition~\ref{cosetfactorization}, the number of distinct monic irreducible factors of $x^n-1$ over $\F_q$ is equal to the number of $q$-cyclotomic cosets modulo $n$. Let
\[
A_d=\left\{i\in\Z_n : \frac{n}{\gcd(i,n)}=d \right\}.
\]
Since $A_d$ is empty if $d\nmid n$, we will assume that $d$ divides $n$,
in which case $\# A_d=\varphi(d)$. We claim that every $q$-cyclotomic coset contained in $A_d$ has
cardinality $\ord_d(q)$. Indeed, if $i\in A_d$, then the cardinality of $C_i$ is the smallest positive integer $r$ such that $iq^r \equiv i\pmod{n}$. This is equivalent to $n\mid i(q^r-1)$. Writing $i=\frac{n}{d}u$, with $\gcd(u,d)=1$, we see that this condition is equivalent to $d\mid q^r-1$. The smallest such $r$ is precisely $\ord_d(q)$. Therefore $A_d$ is a disjoint union of $\ind_d(q)$ distinct $q$-cyclotomic cosets
and formula \eqref{no of irred fact of cyclo} is obtained.
The term corresponding to \(d=1\) accounts for the cyclotomic coset $C_0=\{0\}$, with the convention $\ord_1(q)=1$.
\end{proof}
We note that $x^n-1=\prod_{d\mid n}\Phi_d(x)$, where $\Phi_d(x)$ is the $d$th cyclotomic polynomial. The $d$th cyclotomic
polynomial factorizes into $\ind_{d}(q)$ irreducibles over $\F_q$, provided $q$ and $d$ are coprime.

As each irreducible factor of $x^n - 1$ over ${\F }_q$ corresponds to a cyclic code of
length $n$ over $\F_q$ (see, for example, the books of van Lidl and Niederreiter \cite{LidlNiederreiter1, LidlNiederreiter2}), determining all irreducible factors of $x^n - 1$
over $\F_q$ is of paramount importance. Indeed, many papers are devoted to this factorization. A fairly
general result in this direction can be found in Brochero-Martinez et al.\,\cite{BGO}. For more on the connection with coding theory, see for example \cite{BJ,BJU}.

For a polynomial $g(x)=c_0+c_{1}x+\dotsm+c_{n-1}x^{n-1} \in \F_{q}[x]$, we set $$\bar{g}(x)=c_{n-1}+c_{n-2}x+\dotsm+c_{0}x^{n-1}.$$ Note that $\bar{g}(x)=x^{n-1} g(\tfrac{1}{x})$ and $\overline{(\bar{g}(x))}=g(x)$. Then the reciprocal polynomial of $g(x)$ is $g^{*}(x)=\bar{g}(x)/g(0)$. In particular, if a polynomial is equal to its reciprocal polynomial over $\F_{q}$, then it is called \emph{self-reciprocal} over $\F_{q}$.

\vspace{1mm}
By the uniqueness of the factorization of a polynomial into irreducible factors, we can express equation \eqref{irred fact of cyclo eq} in the form  
\[
x^{n} - 1 = \delta \cdot g_{1}(x) \dotsm g_{s}(x) \cdot h_{1}(x) \cdot h_{1}^{*}(x) \dotsm h_{r}(x) \cdot h_{r}^{*}(x),
\]  
where $\delta \in \F_q^{\times}$, the polynomials $g_1, \dotsc, g_s$ are self-reciprocal irreducible polynomials, and the pairs $h_i(x), h_i^*(x)$ consist of the remaining irreducible factors, grouped together with their reciprocal polynomials.

\begin{prop}\label{reciprocal poly prop}
Let $n\geq 3$ be odd. Given a  factorization of $x^n-1$ into monic irreducible
polynomials over $\F_2$, one can compute
$\mc{O}(n)$ as
\[
\mc{O}(n)=i_{2}^{\mathrm{odd}}(n)+S_2(n),
\]
where $i_{2}^{\mathrm{odd}}(n)$ is the number of unordered pairs
$\{f,f^*\}$ of distinct reciprocal irreducible factors with
$\deg f$ odd, and $S_2(n)$ is the number of self-reciprocal
irreducible factors $f$ satisfying
\[
\deg f\equiv 2\pmod{4}.
\]
In particular, the factor $x-1$ contributes to neither term.
\end{prop}

\begin{proof}
Each nonzero $2$-cyclotomic coset $C_i$ modulo $n$ corresponds to an
irreducible factor $f_i$ of $x^n-1$ over $\F_2$, with
$\deg f_i=|C_i|$, while $f_i^*$ corresponds to $C_{-i}$. If
$C_i\neq C_{-i}$, then, after identifying $i$ with $-i$, the pair
$\{C_i,C_{-i}\}$ gives one cycle of $G_n$ of length $|C_i|$; this
cycle is odd precisely when $\deg f_i$ is odd. If $C_i=C_{-i}$, then
$f_i$ is self-reciprocal and the corresponding cycle of $G_n$ has
length $|C_i|/2$; it is odd precisely when $|C_i|\equiv 2\pmod{4}$.
These two cases account for all odd cycles of $G_n$. Finally, $x-1$
corresponds to the zero cyclotomic coset and has degree $1$, so it is
not counted.
\end{proof}

\noindent
From \eqref{no of irred fact of cyclo}, it follows that
\[
i_2(n)=1+\sum_{\substack{d\, \mid \, n,\ d>1\\2\mid\ord_d(2)}}
\ind_d(2)+ \sum_{\substack{d\, \mid \, n,\ d>1\\2\, \nmid\,\ord_d(2)}}\ind_d(2),
\]
where the term \(1\) 
is the contribution from the factor \(x-1\). Let \(C_i\) be a \(2\)-cyclotomic coset modulo \(n\), and put $d=\frac{n}{\gcd(i,n)}$. Then $C_i=C_{-i}$ if and only if there exists an integer \(j\geq1\) such that $2^j\equiv-1\pmod{d}$. When \(\ord_d(2)\) is even, this is equivalent to
$2^{\ord_d(2)/2}\equiv-1\pmod{d}$. 
When \(\ord_d(2)\) is odd, the equation $2^j\equiv-1\pmod{d}$ does not have a solution.
Thus Proposition
\ref{reciprocal poly prop} gives, for every odd integer \(n\geq3\),
\[
\mc{O}(n)
=\sum_{\substack{d\, \mid \, n,\ d>1\\2\, \nmid\,\ord_d(2)}}
\ind_d(2)+\sum_{\substack{d\, \mid \, n,\ d>1\\
\ord_d(2)\equiv2\pmod{4}\\
2^{\ord_d(2)/2}\equiv-1\pmod{d}}}
\ind_d(2).
\]
This reproves Theorem \ref{thm-FLS} and Proposition \ref{prop:ord4-formula}.

\subsection{A reduction lemma and its applications} 
We next use the prime factorization of $n$ to identify the prime-power factors that do not contribute to $\mc{O}(n)$. To do so, we 
require a trivial lemma. 
\begin{lem}
\label{lem:CRTord}
Let $d$ be odd and $p_1,\dotsc,p_k$ its prime divisors. Then
$$\ord_d(2^a)\equiv \lcm(\ord_{p_1}(2^a),\dotsc,\ord_{p_k}(2^a))\pmod{2},$$
where $a$ is any positive integer.
\end{lem}
\begin{proof}
Note that if $d=p_1^{e_1}\cdots p_k^{e_k}$, then  
$\ord_d(2^a)=\lcm(\ord_{p_1^{e_1}}(2^a),\dotsc,\ord_{p_k^{e_k}}(2^a))$
by the Chinese Remainder Theorem. The proof is completed on noting that $\ord_{p^k}(2^a)$ divides $p^{k-1}\ord_p(2^a)$.
\end{proof}

Proposition \ref{prop:ord4-formula}  together with Lemma \ref{lem:CRTord} shows that prime factors $p$ for which $\ord_p(4)$ is even do not contribute to $\mc{O}(n)$. This gives the following useful reduction.
\begin{lem}\label{lem-even}
Let $n\ge 3$ be an odd integer and $m$ its largest divisor having only
prime factors $p$ for which $\ord_p(4)$ is odd. Then
$\mc{O}(n)=\mc{O}(m)$.
Moreover, $\mc{O}(n)=0$ if and only if $m=1$, that is if and only if $\ord_p(4)$ is even for every prime divisor $p$ of
$n$.
\end{lem}

As a consequence of Lemma \ref{lem-even}, we recover the following characterization of Wu and Fu~\cite{WuFu}.
\begin{thm}[Theorem 1.1 of \cite{WuFu}]
Let $k\geq 2$.
\begin{enumerate}
    \item A tight equi-difference conflict-avoiding code of length $2^k-1$ and weight $3$ exists if and only if $k$ is a power of $2$.
    \item A tight equi-difference conflict-avoiding code of length $2^k+1$ and weight $3$ exists if and only if $k$ is even.
\end{enumerate}
In either case, whenever such a code exists, it has $2^{k-2}$ codewords. The exceptional case $k=1$ in the second family gives the tight code of length $3$.
\end{thm}
\begin{proof}
Suppose first that $k=2^s$, with $s\geq 1$, and put $n=2^k-1$. Then $3\mid n$, and, in fact, $\nu_3(n)=1$. If $p\mid n$ and $p\neq 3$, then $\ord_p(2)\mid k=2^s$. Since $\ord_p(2)\neq 1,2$, it follows that $4\mid \ord_p(2)$, and hence $\ord_p(4)$ is even. Thus, Lemma~\ref{lem-even} gives $\mc{O}(n)=\mc{O}(3)=1$, so Proposition~\ref{tight equi-diff prop} yields a tight equi-difference code.

Conversely, suppose that $k$ is not a power of $2$. If $k>1$ is odd, then $3\nmid 2^k-1$ and $2^k-1\equiv 3\pmod{4}$. Hence $2^k-1$ has a prime divisor $p\equiv 3\pmod{4}$. Since $\ord_p(4)\mid \frac{p-1}{2}$, and thus $\ord_p(4)$ is odd. By Lemma~3, $\mc{O}(2^k-1)>0$. If $k=2^s m$ with $s\geq1$ and $m>1$ odd, then $2^m-1\mid 2^k-1$, while $2^m-1\equiv3\pmod{4}$ and $3\nmid2^m-1$. Thus there is a prime $p\neq3$, $p\mid2^k-1$, with $\ord_p(4)$ odd. Since $3\mid2^k-1$, the divisors $3$ and $p$ both give positive contributions in~\eqref{O(n) defn interms ord of 4}, so $\mc{O}(2^k-1)>1$.

Now put $n=2^k+1$. Write $k=2^s m$ with $m$ odd. For any prime $p\mid n$, if $r=\ord_p(2)$, then $r\mid 2k,\ r\nmid k$. If $k$ is even, then $s\geq1$, and these two conditions force $\nu_2(r)=s+1$. Hence $\ord_p(4)=r/2$ is even for every $p\mid n$. Lemma~\ref{lem-even} gives $\mc{O}(n)=0$, and since $3\nmid n$, Proposition~\ref{tight equi-diff prop} gives a tight code.

If $k>1$ is odd, then $\nu_2(r)=1$ for every $p\mid n$, so $\ord_p(4)$ is odd. Moreover $3\mid n$. If $n$ has a prime divisor $p\neq3$, then the divisors $3$ and $p$ both contribute positively
to~\eqref{O(n) defn interms ord of 4}, and hence $\mc{O}(n)>1$. If $n=3^\ell$, then $\ell\ge2$. In this case, the divisors $3$ and $9$ both contribute positively to \eqref{O(n) defn interms ord of 4}, since $\ord_3(4)=1$,  $2\equiv-1\pmod3$, and $\ord_9(4)=3$, $2^3\equiv-1\pmod9$. Thus again $\mc{O}(n)>1$.

Finally, count the nonzero differences covered by the code. For $n=2^k+1$, every codeword contributes four differences, so $4|\mathcal{C}|=n-1=2^k$. For $n=2^k-1$, there is one exceptional codeword contributing only two differences, while every other codeword contributes four. Hence $2+4\bigl(|\mathcal{C}|-1\bigr)=n-1=2^k-2$. Thus, in either case,
$|\mathcal C|=2^{k-2}$.
\end{proof}

\subsection{Counting odd-length orbits via the Cauchy--Frobenius identity}
Let $a > 1$ and $m \geq 1$ be integers such that $\gcd(a, m) = 1$. On the additive group $\Z_m$ 
we define the automorphism $\alpha \in \mathrm{Aut}(\Z_m)$ by $\alpha(x) = a x$ for all $x \in \Z_m$. Let $n  \coloneq \ord_m(a)$ denote the multiplicative order of $a$ modulo $m$, so that $\langle \alpha \rangle \cong C_n$ is a cyclic group of automorphisms of $\Z_m$ of order $n$.

We let $G  \coloneq  \langle \alpha \rangle$ act on $\Z_m$ in the natural way: for $x \in \Z_m$ and $k \in \Z$,
\[
\alpha^k \cdot x  \coloneq \alpha^k(x) =  a^k x.
\]
This action partitions $\Z_m$ into $G$-orbits. Reasoning as in the proof of Lemma \ref{noofirrfactors} one sees that
\[
i_a(m) = \sum_{d \mid m} \ind_d(a).
\]
It is also easy to  see that the number of odd orbits is equal to \[
i_a^{odd}(m) = \sum_{d \mid m ,\ \ord_d(a)\ \textrm{is odd}} \ind_d(a).
\]

Having expressed $i_a^{\mathrm{odd}}(m)$ explicitly for a general multiplier $a$, we can specialize to $a=2$ and relate this quantity to $\mc{O}(m)$, the two invariants agree exactly outside a correction term $S_2(m)$ accounting for self-reciprocal polynomials (see Proposition \ref{reciprocal poly prop}).
The total number $i_a(m)$ of orbits of the multiplication map
$x\mapsto ax$ modulo $m$ is a classical quantity; for a recent discussion of this
orbit-counting problem, its relation with the doubling map, and several
of its applications, see Allouche, Stipulanti and Yao~\cite{ASY}.
Here we require a refinement in which only the orbits of odd length are
counted. In fact, if $i_a^{\mathrm{odd}}(m)$ denotes their number, then
\[
    i_a^{\mathrm{odd}}(m)=2i_a(m)-i_{a^2}(m),
\]
as follows by comparing the actions of $\langle \alpha \rangle$ and
$\langle \alpha^2\rangle$. Applying the Cauchy--Frobenius identity gives the
following explicit formula for it.
\begin{thm}
\label{thm:Cauchy-Frobenius}
Let $a > 1$ and $m \geq 1$ with $\gcd(a, m) = 1$. Then the number of $G$-orbits of odd length in the action of $G = \langle \alpha \rangle$, where $\alpha(x) = ax$, on the additive group $\Z_m$, is given by
\[
 i_a^{\mathrm{odd}}(m) =\frac{\kappa_a(m)}{\ord_m(a)} \Bigl( \sum_{\substack{d \mid \ord_m(a) \\ d\ \textrm{is odd}}} \varphi\Bigl(\frac{\ord_m(a)}{d}\Bigr)(m, a^d - 1) \Bigr),
\]
where $\kappa_a(m) =1$ if $\ord_m(a)$ is odd and $\kappa_a(m) =2$ otherwise.
\end{thm}
\begin{proof}
To isolate the contribution of \emph{odd-length orbits}, we consider the subgroup $H  \coloneq  \langle \alpha^2 \rangle \leq G$, which consists of the even powers of $\alpha$. This subgroup acts on $\Z_m$ via
\[
\alpha^{2k} \cdot x  \coloneq  a^{2k} x.
\]
We observe that every $G$-orbit of odd length remains a single $H$-orbit, while every $G$-orbit of even length splits into two $H$-orbits of equal size.
Let $i_a^{\mathrm{even}}(m)$ denote the number of $G$-orbits of even length. Then the number of $H$-orbits is given by 
\[
|H\setminus \Z_m| = i_a^{\mathrm{odd}}(m) + 2 i_a^{\mathrm{even}}(m).
\]
Solving for $i^{odd}_{a}(m)$, we deduce that
\[  i_a^{\mathrm{odd}}(m) = 2 \times |G\setminus \Z_m| - |H\setminus \Z_m|. \]
If $\ord_m(a)$ is odd, then $i_a^{\mathrm{odd}}(m)= |G\setminus \Z_m|$. 
If $\ord_m(a)$ is  even, then by Deaconescu’s identity \cite{Deaconescu}, we have

\begin{align*}
    	i_a^{\mathrm{odd}}(m) &=	 \frac{2}{\ord_m(a)} \Bigl( \sum_{\substack{d \mid \ord_m(a) }}\!\! \varphi\Bigl(\frac{\ord_m(a)}{d}\Bigr)(m, a^d - 1) -
    	\!\!\sum_{\substack{d \mid  \frac{\ord_m(a) }{2}}} \varphi\Bigl(\frac{\ord_m(a)}{2d}\Bigr)(m, a^{2d} - 1)\Bigr)\\
    	&=\frac{2}{\ord_m(a)} \Bigl( \sum_{\substack{d \mid \ord_m(a) \\ d\  \textrm{is odd}}} \varphi\Bigl(\frac{\ord_m(a)}{d}\Bigr)(m, a^d - 1) \Bigr).
    	\tag*{\qed}
\end{align*}
\renewcommand{\qed}{}
\end{proof}

This closed form is recorded below as a corollary, expressing $\mc{O}(n)$ solely in terms of $\ord_n(2)$ and its divisors, in the same computational spirit as Deaconescu's identity used earlier for $i_2(n)$.

\begin{cor}\label{cor:On-explicit}
For every odd integer $n>1$,
\[
\mc{O}(n) \;=\; \frac{1}{2}\bigg[\,\frac{x_2(n)}{\ord_n(2)} \sum_{\substack{d \,\mid\, \ord_n(2) \\ d\ \mathrm{odd}}} \varphi \Bigl(\frac{\ord_n(2)}{d}\Bigr)(n, 2^d - 1) \;-\; 1\,\bigg] + S_2(n),
\]
where $S_2(n)$ is defined as in Proposition \ref{reciprocal poly prop}.
\end{cor}

\section{Applying techniques used in the study of the multiplicative order: initial considerations}\label{initial consid sec}
The description of $\mc{O}(p)$ obtained in Corollary~\ref{Opequivalence} naturally leads to questions concerning primes for which $4$ has prescribed multiplicative order. We therefore recall some results from the theory of near-primitive roots that will be used in the proof of Theorem~\ref{main} and later in our study of the level sets of $\mc{O}(p)$ in arithmetic progressions.

Given a rational number $g$ and a prime $p$ such that $\nu_p(g)=0$, we denote by $\ord_p(g)$ the \emph{multiplicative order}
of $g$ modulo $p$ as defined in Definition~\ref{mult ord def}.
Given an integer $t > 1$, we say 
that $g$ is a \emph{near-primitive root of index $t$} if $\nu_p(g) = 0$ and $\ord_p(g)=(p-1)/t$.
Notice that this entails that $p\equiv 1\pmod{t}$.
Wagstaff \cite{Wagstaff} showed that, assuming GRH,
this set of primes has density
\begin{equation}
\label{classicalwagstaff} 
\delta_g(t)=\sum_{n=1}^\infty
\frac{\mu(n)}{[\Q(\zeta_{nt},g^{1/nt}) : \Q]},
\end{equation} 
where $\zeta_m=e^{2\pi i/m}$.
Actually, for the proof of \eqref{classicalwagstaff} it is enough to assume that the Dedekind zeta-functions of all fields appearing in \eqref{classicalwagstaff} with $n$ squarefree all have all their non-trivial zeros
on the half-line. The proof is an easy generalization of a proof of Hooley \cite{Hooley1}, who dealt with \emph{primitive roots} ($t=1$ case) and evaluated
$\delta_g(1)$.
Wagstaff evaluated $\delta_g(t)$ explicitly in the 
form $\rho(g)A$, with $\rho(g)$ a rational number.
He excluded the case where $g$ is minus a square, which was later 
dealt with by Moree \cite{Mnear}, whose result is in complete agreement with the determination of
this density using a very different, much more algebraic, approach, see 
Lenstra et al.\,\cite{LMS}.

\par Given two coprime positive integers $a$ and $d,$ we put 
\[
\label{cqda}
Q(d,a)=\{p>2 : \ p\equiv a \pmod{d},\ \ord_p(4)=(p-1)/2\}.
\]
These sets provide useful concrete examples of the general near-primitive-root density formula. They will also play a role in Section~\ref{Mult ord sec} since the condition $\mc{O}(p)=1$ is characterized by $\ord_p(4)=(p-1)/2$, together with the requirement that this order be odd.

The order condition defining $Q(d,a)$ also occurs naturally in the study
of digit expansions (cf.\,Remark \ref{rem:Girstmair}). If $p\in Q(d,a)$, then the base-$4$ expansion of
$1/p$ has period length $(p-1)/2$. Girstmair \cite{G3} studies the digit frequencies of the
base-$b$ expansion of $1/p$ precisely under the hypothesis
\[
  \ord_p(b)=\frac{p-1}{2},
\]
expressing these frequencies in terms of generalized Bernoulli numbers and determining when they reduce to class numbers of imaginary quadratic fields. In particular, $b=4$ belongs to his class-number cases. Thus the
primes counted by $Q(d,a)$ also parametrize a family of base-$4$ expansions whose digit frequencies admit explicit arithmetic descriptions.

The density $\delta(Q(d,a))$ of these sets was determined under GRH in 
\cite[Prop.\,1.12]{HKMS}. Before applying these ideas to the general level sets of $\mc{O}(p)$, we discuss two elementary instances of the density calculation, namely
$d=1$ and $d=8$.
For this, and our proof of 
Theorem \ref{main}, we need a small variation of Wagstaff's result, namely that 
given a number field $K$, under GRH, the set
$$\bigl\{p\textrm{~splits~completely~in~}K\textrm{~and~}
\ord_p(g)=(p-1)/t\bigr\}$$
has density 
\begin{equation}
\label{wagstaffK} 
\delta_{K,g}(t)=\sum_{n=1}^\infty
\frac{\mu(n)}{[K(\zeta_{nt},g^{1/nt}) : \Q]}.
\end{equation} 
\par By Wagstaff's work \cite{Wagstaff}, under GRH, we have $\delta(Q(1,1))=\delta_4(2)=3A/2$. It follows by \eqref{wagstaffK} with $K=\Q$, $g=4$ and $t=2$ (or \eqref{classicalwagstaff} with
$g=4$ and $t=2$), that 
\begin{equation}
\label{startup}
\delta(Q(1,1))=\sum_{n=1}^\infty
\frac{\mu(n)}{[\Q(\zeta_{2n},4^{1/2n}) : \Q]}=
\sum_{n=1}^\infty\frac{\mu(n)}{[\Q(\zeta_{2n},2^{1/n}) : \Q]}.
\end{equation}
The following lemma (cf.\,\cite[Lemma 1]{MoreeI}) allows one to determine the field degrees in 
this sum and rests on the observation that $\sqrt{2}\in \Q(\zeta_n)$ if and only if $8\mid n$.
\begin{lem}
\label{lem:graad}
Let $m\mid n$. We have
\[
[\Q(\zeta_n,2^{1/m}) : \Q]=\begin{cases}
\varphi(n)m/2 &\ \textrm{if}\ 2\mid m\ \textrm{and}\ 8\mid n;\cr
\varphi(n)m & \textrm{~otherwise}.
\end{cases}
\]
\end{lem}
\noindent 
We infer that $[\Q(\zeta_{2n},2^{1/n}):\Q]=\varphi(2n)n$ if $4\nmid n$, 
which in
combination with \eqref{startup} yields
$$
\delta(Q(1,1))=\sum_{n=1}^{\infty}\frac{\mu(n)}{\varphi(2n)n}.
$$
By considering the odd and even $n$ separately, one obtains 
$$
\delta(Q(1,1))=\sum_{\substack{n=1\\2\nmid n}}^{\infty}\frac{\mu(n)}{\varphi(2n)n}+\sum_{m=1}^{\infty}\frac{\mu(2m)}{\varphi(4m)2m}=\frac{3}{4}\sum_{\substack{n=1\\2\nmid n}}^{\infty}\frac{\mu(n)}{\varphi(n)n}.$$
Using that 
\begin{equation}
\label{EulerP}    
\sum_{\substack{n=1\\ (j,n)=1}}^{\infty}\mu(n)f(n)=\prod_{p\nmid j}(1-f(p))  
\end{equation}
holds certainly true if the sum is absolutely convergent and $f(n)$ is a multiplicative function defined on the squarefree integers, we
then get
$$
\delta(Q(1,1))=\frac{3}{4}\prod_{p>2}\Big (1-\frac{1}{p(p-1)} \Big)
=\frac{3}{2}A.
$$
This backside-of-the-envelope calculation simplifies the calculation at
\cite[pp.\,72--73]{HKMS} even further (where for 
$[\Q(\zeta_{2n},2^{1/2n}) : \Q]$ in the 
final line of page 72 one should read  $[\Q(\zeta_{2n},2^{1/n}) : \Q]$).
\par By the same heads-on approach we now determine $\delta(Q(8,j))$ with $j$ odd.
\begin{lem}
\label{lem:Q8}
Let $j$ be an odd integer.
If $j\equiv 1\pmod{8}$, then $Q(8,j)$ is empty, otherwise
under GRH we have 
$\delta(Q(8,j))=A/2$.
\end{lem}
\begin{proof}
If $p\equiv 1\pmod{8}$, then $\ord_p(2)\mid \frac{p-1}{2}$ by \eqref{QR}. Since $\frac{p-1}2$ is even,
it follows that $\ord_p(4)$ divides $(p-1)/4$ and hence $Q(8,1)$ is empty.
We see that $Q(8,5)=Q(4,1)$, the density of which we compute by setting 
$K=\Q(i)$, $g=4$ and $t=2$ in \eqref{wagstaffK}. Since
$\Q(\zeta_a,\zeta_b)=\Q(\zeta_{\lcm(a,b)})$,
we conclude by Lemma \ref{lem:graad} that $[\Q(i,\zeta_{2n},2^{1/n}) : \Q]=
\varphi(\lcm(4,2n))n$ if $4\nmid n$.
It follows that
$$\delta(Q(8,5))=\delta(Q(4,1))=\frac{1}2\sum_{n=1}^{\infty}\frac{2\mu(n)}{\varphi(\lcm(4,2n))n}=\frac{A}{2}.$$
In the final step we used that the argument in the sum is a multiplicative function in
$n$ and applied \eqref{EulerP} with $j=1$. 
\par The primes $p$ that split completely in
$\Q(\sqrt{2})$ are precisely the primes $p$ with
$\legendre{2}{p}=1$, that is satisfy
$p\equiv \pm 1\pmod{8}$. Recalling that $Q(8,1)$ is empty and using Lemma \ref{lem:graad} it now follows that
$$\delta(Q(8,7))=\sum_{n=1}^{\infty}\frac{\mu(n)}{[\Q(\sqrt{2},\zeta_{2n},2^{1/n}) : \Q]}=\frac{1}{2}\sum_{n=1}^{\infty}\frac{\mu(n)}{\varphi(n)n}=\frac{A}2.$$
Finally,
\[
\delta(Q(8,3))=\delta(Q(1,1))-\delta(Q(8,5))-\delta(Q(8,7))=\frac{3}{2}A-\frac{A}2-\frac{A}2=\frac{A}2.\tag*{\qed}
\]
\renewcommand{\qed}{}
\end{proof}
\begin{rem}
These computations allow one to check once more that, under GRH, $\delta_1=A$ (as claimed in Theorem \ref{main}). Under that assumption we namely
have $$\delta_1=\delta(Q(4,3))=\delta(Q(8,3))+\delta(Q(8,7))=A/2+A/2=A.$$
\end{rem}
\subsection{An application}
Rodier \cite{Rodier} in the context of extended BCH codes established the following result.
\begin{thm}
Suppose that $p\equiv 7\pmod{8}$ and $\ord_p(2)=(p-1)/2$. Put
$$S_m=\sum_{\alpha \in \F_{2^m}}(-1)^{\ \operatorname{Tr}(\alpha^p)},$$
where$$\ \operatorname{Tr}(x)\coloneq \sum_{i=0}^{m-1}x^{2^i}$$
is the trace of $\F_{2^m}$ over $\F_2$.
We have $$\limsup_{\substack{m\rightarrow \infty\\ 2\, \nmid\, m}}\frac{|S_m|}{2^{m/2}}=p-1.$$
\end{thm}
Note that the primes $p$ for which the theorem applies form the set $Q(8,7)$ and
by Lemma \ref{lem:Q8} have density $A/2$ under GRH. Rodier outlined a
different proof of this.
\section{Proof of Theorem \ref{main}}\label{Pf of thm 2 sec}
\emph{Proof}. By Corollary \ref{Opequivalence} the set $P_k$ consists of the primes $p\equiv 1\pmod{2k}$, for which $\ord_p(4)=(p-1)/2k$ and $\ord_p(4)$ is odd.

If $k \equiv 2 \pmod{4}$, then $(p-1)/2k$ is odd implies that $p\equiv 5\ \pmod{8}$ and $\ord_{p}(4)=(p-1)/2k$ implies that $\ord_{p}(2) \mid \frac{p-1}{2}$ and hence $\legendre{2}{p}=1$. By the supplementary law of quadratic reciprocity \eqref{QR}, the conditions $\legendre{2}{p}=1$ and $p\equiv 5\pmod{8}$ cannot be satisfied at the same time. Hence
$P_k$ is empty.\\
\indent For the remainder of the proof we assume GRH. Then it follows from \cite[Theorem 3.1]{HKMS} that the number of odd primes $p\leq x$ such that $\ord_p(4)=(p-1)/2k$ and $p\equiv 1 \pmod{f}$, where $f$ is an arbitrary natural number, is given by
\[ \delta(f,k)\frac{x}{\log x} + O_{k,f}\Big(\frac{x\log\log x}{\log^2x} \Big) \]
where
\[ \delta(f,k)\coloneq \sum_{n=1}^{\infty}\frac{\mu(n)}{[\Q(\zeta_f,\zeta_{2nk},2^{1/nk}) : \Q]}. \]

Let $p\equiv 1\pmod{2k}$ be a prime. Put $r=\nu_2(2k)$. Note that $(p-1)/2k$ is odd if and only if $p\equiv 1+2^r \pmod{2^{r+1}}$. We conclude that 
\begin{equation}
\label{asymptoticforPk}    
P_k(x) = \delta_k \frac{x}{\log x} + O_{k}\Big(\frac{x\log\log x}{\log^2x} \Big)
\end{equation}
where $\delta_k=\delta(2^r,2k)-\delta(2^{r+1}, 2k)$. 

From now on we assume that $k \not\equiv 2 \pmod{4}$. It is easy to see that \begin{equation*}
\delta_k
=\sum_{\substack{n=1\\ 
2\, \nmid\, n}}^{\infty}\Bigl(\frac{\mu(n)}{[\Q(\zeta_{2nk}, 2^{1/nk}) : \Q]}-\frac{\mu(n)}{[\Q(\zeta_{4nk}, 2^{1/nk}) : \Q]}\Bigr).
\end{equation*}
For $n$ odd and $k \not\equiv 2 \pmod{4}$ it follows from Lemma \ref{lem:graad} that 
$$[\Q(\zeta_{4nk}, 2^{1/nk}) : \Q]=2[\Q(\zeta_{2nk}, 2^{1/nk}) : \Q]$$ and
hence we obtain 
\begin{equation}
\label{tweede}
\delta_k=\delta(2^r,2k)-\delta(2^{r+1},2k)=\sum_{\substack{n=1\\ 
2\, \nmid\, n}}^{\infty}\frac{\mu(n)}{[\Q(\zeta_{4nk}, 2^{1/nk}) : \Q]}.
\end{equation}
On invoking Lemma \ref{lem:graad} again, we obtain
\begin{equation}\label{k case tweede}
\sum_{\substack{n=1\\ 2\, \nmid\, n}}^{\infty}\frac{\mu(n)}{[\Q(\zeta_{4nk}, 2^{1/nk}) : \Q]}=
\begin{cases}
2 S(k) \ \textrm{if } 2 \mid k;\\[1ex]
S(k) \ \textrm{if } 2 \nmid k.
\end{cases}
\end{equation}
where
\[ 
S(k) \coloneq\sum_{\substack{n\geq 1\\2\nmid n}} \frac{\mu(n)}{\varphi(4nk)nk}. 
\]
We have
$$S(k)=\frac{1}{k\varphi(4k)}\sum_{\substack{n\geq 1\\2\nmid n}} \frac{\mu(n)\varphi(4k)}{\varphi(4nk)n}=\frac{1}{k\varphi(4k)}\prod_{\substack{p\mid k\\p>2}}\Bigl(1-\frac{1}{p^2}\Bigr) \prod_{\substack{p\nmid k\\p>2}} \Bigl(1-\frac{1}{p(p-1)}\Bigr),$$
where we used that the argument of the sum is a multiplicative function in $n$.
Recall the definition \eqref{def:Ek} of $E(k)$.
Note that
$$A\cdot E(k)=\frac{1}{k\varphi(k)}\prod_{p\mid k}\Bigl(1-\frac{1}{p^2}\Bigr) \prod_{p\nmid k} \Bigl(1-\frac{1}{p(p-1)}\Bigr).$$
Comparing the Euler product of $S(k)$ with that of $E(k)$, we see that
$$S(k)=A\cdot E(k)\frac{\varphi(k)}{\varphi(4k)}b\quad\text{~with~}\quad
b=
\begin{cases}
\frac{4}{3} \ \textrm{if } 2 \mid k;\\[1ex]
2 \ \textrm{if } 2 \nmid k.
\end{cases}
$$
The proof is completed on combining  \eqref{asymptoticforPk}, \eqref{tweede}, \eqref{k case tweede} and the latter identity. 
\qed

\section{Level sets of \texorpdfstring{$\mc{O}(p)$}{Op} with \texorpdfstring{$p$}{p} in a prescribed arithmetic progression} \label{Mult ord sec}
\par Given a prescribed arithmetic progression $a\pmod{d}$ and any integer
$k\ge 0$, recall that
$$P_{k}(a,d)\coloneq \{p\equiv a\pmod{d} :~\mc{O}(p)=k\}.$$
The case $k=0$, which does not fall within the near-primitive-root framework used below, is treated separately in Section~\ref{sec:k-zero}. We henceforth assume that $k\geq 1$.
\begin{thm}[under GRH] 
\label{thmgeneral}
Let $a,d,k$ be positive integers.
The set $P_k(a,d)$ has a natural density $\delta_k(a,d)$, and $P_k(a,d)(x)$, the number of primes $p\leq x$ in $P_k(a,d)$, satisfies
\[P_k(a,d)(x)=\delta_k(a,d)\frac{x}{\log x} + O_{k,d}\bigg( \frac{x\log\log x}{(\log x)^2} \bigg). \]
\end{thm}

\begin{proof}
We may assume that $a$ and $d$ are coprime, since otherwise 
the result is trivially true (with $\delta_k(a,d)=0$).
Note that 
\[
P_{k}(a,d)=\Bigl\{p\equiv a\pmod{d} :~p\equiv 1+2k\pmod{4k}\ \textrm{ ~and~}\ord_p(4)=\frac{p-1}{2k}\Bigr\}.
\]
In the notation of Hu et al.~\cite{HKMS}, let
\[
\mathcal{P}(g,t,d,a)=\left\{p:~p\equiv1\pmod{t},\quad p\equiv a\pmod{d},\quad \ord_p(g)=\frac{p-1}{t} \right\}.
\]
Put $D=\lcm(d,4k)$. If $a\not\equiv 1+2k\pmod{(d,4k)}$, then the two congruence conditions in the definition of
$P_{k}(a,d)$ are incompatible and $P_k(a,d)$ is an empty set. Otherwise, let $b\pmod{D}$ be the unique residue class determined by
\[
b\equiv a\pmod{d}, \qquad b\equiv 1+2k\pmod{4k}.
\]
Then $P_k(a,d)=\mathcal P(4,2k,D,b)$. It now follows from \cite[Thm.~3.1]{HKMS} that, under GRH, the number of primes $p\leq x$ in $P_k(a,d)$ is given by
\[ \delta(4,2k,D,b)\frac{x}{\log x} + O_{k,d}\bigg( \frac{x\log\log x}{(\log x)^2} \bigg), \]
where $\delta(4,2k,D,b)$ is the natural density of the set $\mathcal P(4,2k,D,b)$, and hence we deduce
\[
\delta_k(a,d)=\delta(4,2k,D,b).\qedhere
\]
\end{proof}

We remark that \cite[Thm.\,3.1]{HKMS} expresses $\delta(4,2k,D,b)$ as an infinite sum involving
certain Galois theoretic intersection coefficients. We do not formulate that result here, as
for us the simpler variant \eqref{wagstaffK} is enough.
\par Note that $\delta_k(a,d)$ does not change if we 
replace $P_k(a,d)$ by a set that only differs in finitely
many primes from it.
Thus in particular $\delta_k(1,1)=\delta_k(1,2)$, which for ease of notation we 
denote by $\delta_k$. 
\subsection{The case \texorpdfstring{$k=0$}{k=0}}\label{sec:k-zero}
Here we are interested in counting the primes $p$ for which
$\ord_p(4)$ is even. 
Counting the primes $p$ for which,
for given integers $d$ and $g$ we have $d\mid \ord_p(g)$, is
a well-studied problem. In this case the density can be shown to exist 
unconditionally. It is rather easy to
determine and turns out to be a rational number, see for example \cite{MoreeFACM,Wie1,Wie2}.  
\par The case $d=2$ is special. 
For an odd prime $p\nmid g$, the multiplicative order $\ord_p(g)$ is even if and only if $g^k \equiv -1 \pmod{p}$ for some $k\geq 1$, or equivalently, if and only if $p$ divides some term of the sequence $(g^k+1)_{k\ge 1}$.
Recall that
an integer $n$ divides a sequence of integers $(a_k)_{k\ge 1}$ if it divides at least one non-zero term.
Thus, questions concerning the evenness of the multiplicative order may be viewed as questions about the prime divisors of this sequence. The special case $g=2$, corresponding to the sequence $(2^k+1)_{k\ge 1}$,
has been studied particularly extensively. For an informal description of the method used here, see the survey by Moree \cite{Msurvey}.
Integers dividing the sequence $(a^k+b^k)_{k\ge 1}$ play an important role in coding theory, see \cite{Jitman,Jitman2}.

\begin{prop} \label{O(n) zero thm}
Let $p$ be an odd prime. We have $\mc{O}(p)=0$ if and only if
$p$ divides $(4^k+1)_{k\ge 1}$.
Let $m>1$ be a positive integer. If $m$ divides $(4^k+1)_{k\ge 1}$, then $\mc{O}(m)=0$.
\end{prop}
\begin{proof}
As remarked in the introduction of this section, an odd prime $p$ divides $(4^k+1)_{k\ge 1}$ if and only if $\ord_p(4)$ is even, which is equivalent with  $\mc{O}(p)=0$.
All prime divisors $p$ of $m$ have even $\ord_p(4)$.
Therefore, by Lemma \ref{lem-even}, we have  $\mc{O}(m)=0$.
\end{proof}
\begin{rem}
This proof shows that if $m$ divides $(4^k+1)_{k\ge 1}$ for some $k\geq 1$, then $\ord_p(4)$ is even for every prime $p\mid m$. The converse, however, does not hold. Indeed, the condition that
$m$ divide $(4^k+1)_{k \ge 1}$ is equivalent to the existence of an integer $e\geq 1$ such that $\nu_{2}(\ord_{p}(2))=e$ for every prime $p\mid m$ by \cite[Theorem~1]{Moreediv}. On the other hand, the condition
$\mc{O}(m)=0$ requires only that $\nu_{2}(\ord_{p}(2)) \geq 1$ for every prime $p\mid m$. The contrast one also sees asymptotically: the number of integers $m\le x$ dividing 
$(4^k+1)_{k\ge 1}$ grows asymptotically as $C_1x(\log x)^{-2/3}$ (by \cite[Theorem 5]{Moreediv}) and when divided by $\#\{m\le x:\mc{O}(m)=0\}$ the ratio tends to zero by Theorem \ref{density O(m) is zero thm}.
\end{rem}

Recall that the \emph{logarithmic integral} is defined by $\Li(x)=\int_2^x \frac{dt}{\log t}$.
\begin{prop}
\label{512}
Let
$$P_0\coloneq \{p>2:{\mc{O}}(p)=0\}\quad \text{ and }\quad
P_0(x)\coloneq\#\{p\leq x:p\in P_0\}.$$
Then
\begin{equation}
\label{P0counting}    
P_{0}(x)=\frac{5}{12}\Li(x)+O\Bigl(\frac{x(\log\log x)^4}{(\log x)^{3}} \Bigr).
\end{equation}
In particular, the set of primes $p$ for which $\mc{O}(p)=0$ has natural density $\frac{5}{12}$.
\end{prop}
\begin{proof}
Since $\mc{O}(p)=0$ is equivalent to $2\mid\ord_p(4)$, by applying \cite[Thm.\,1, Thm.\,2]{MoreeFACM} with
$d=4$, $g=g_0=2$, $h=1$ and $D(g_0)=8$. We have 
$\gamma=1$ and so $\epsilon_2(4)=\frac{1}{4}$ and hence obtain \eqref{P0counting}.

Alternatively, by applying \cite[Thm.\,2]{Moreediv} with $r=0$, $\psi=4=2^{2^1}$ so that $u=2$, $\lambda=1$, and $\varepsilon=1$, we obtain
\begin{equation}
\label{eengelukje}    
\#\{p\leq x:p>2,\ \ord_p(4)\ \text{is odd}\}=\frac{7}{12}\Li(x)+O\left(\frac{x(\log\log x)^4}{(\log x)^3}\right).
\end{equation}
Taking the complement among the odd primes gives \eqref{P0counting} on using the weak prime number estimate
$\sum_{p\le x}1=\Li(x)+O(x(\log x)^{-3})$.
\end{proof}

Existing results of Moree and Sury \cite{MSury} allow one to quantify how many primes in a prescribed arithmetic progression satisfy 
$\mc{O}(p)=0$.
\begin{thm} \label{dens zero prop}
Let $a$ and $d$ be positive coprime integers, and let
\[
P_{0}(a,d)(x) \coloneq\#\{p\leq x : p\equiv a\pmod{d},\ \mc O(p)=0\}.
\]
Then there exists a rational number $\delta_{0}(a,d)\geq0$ such that
\[
P_{0}(a,d)(x)=\frac{\delta_{0}(a,d)}{\varphi(d)} \Li(x)+O\Bigl(\frac{x\log\log x}{(\log x)^{7/6}}\Bigr),
\]
where $\delta_0(a,d)\geq0$ is rational and can be determined explicitly from Tables~4 and~5 of~\cite{MSury}; In the notation of that paper, one specializes their parameters $a,b$ (which we denote here by $A,B$ to avoid confusion with the residue class $a \pmod d$ to $A=4, B=1$.
\end{thm}
\begin{proof}
By Proposition~\ref{O(n) zero thm}, for every odd prime $p$ we have $\mc{O}(p)=0$ if and only if $p$ divides $4^j+1$ for some $j\ge1$.  Hence $P_0(a,d)(x)$ is exactly the counting function for primes $p\equiv a\pmod d$ dividing at least one term of the sequence $(4^j+1)_{j\ge1}$.  Applying \cite[~Theorem 2]{MSury} with $A=4$ and $B=1$ yields
\[
P_0(a,d)(x)=\delta_{4,1}(a,d)\Li(x)+O\Bigl(\frac{x\log\log x}{(\log x)^{7/6}}\Bigr),
\]
where $\delta_{4,1}(a,d)$ is rational and is determined explicitly by Tables~4 and~5 of that paper.  Setting
$\delta_0(a,d) \coloneq \varphi(d)\delta_{4,1}(a,d)$  gives the asserted formula.  The displayed values of $\delta_0(a,d)$ follow by specializing the tables to $A=4$, $B=1$. 
\end{proof}
Put $\delta \coloneq \nu_{2}(d)$ and $\gamma=\nu_{2}(a-1)$, with the convention that $\gamma= \infty$ when $a=1$. For example, the tables specialize in our case as follows:
\[
\delta_{0}(a,d)
=
\begin{cases}
5/12
& \textrm{if } \delta\leq 1;\\[3pt]
5/6
& \textrm{if } \delta=2 \textrm{ and } \gamma\geq 2;\\[3pt]
0
& \textrm{if } \delta=2 \textrm{ and } \gamma=1.
\end{cases}
\]

\subsection{The case \texorpdfstring{$k=1$}{k=§}}
An odd prime $p$ is said to be \emph{$B$-irregular} if $p$ divides the numerator of at least one of the Bernoulli numbers $B_{2},B_{4},\dotsc,B_{p-3}$, 
and \emph{$B$-regular} otherwise. 
The first fifteen $B$-irregular primes are
\begin{align*}
37, 59, 67, 101, 103, 131, 149, 157, 233, 257, 263,  271,  283, 293, 307.
\end{align*}

Kummer, who introduced the notion of $B$-irregularity, proved that Fermat's Last Theorem is true for 
$B$-regular prime exponents $p$. He did this by relating $B$-regularity to a certain cyclotomic class number indivisibility.

 The \emph{Genocchi numbers} $G_n$ are related to Bernoulli numbers $B_n$ by the formula 
\[
G_n = 2(1-2^n)B_n. 
\]
An odd prime $p$ is said to be \emph{$G$-irregular} if it divides at least one of the integers 
$G_2,G_4,\dotsc,$ $G_{p-3}$, and
\emph{$G$-regular} if it does not.
The first fifteen $G$-irregular primes are  
\begin{align*}
17, 31, 37, 41, 43, 59, 67, 73, 89, 97, 101, 103, 109, 113, 127. 
\end{align*}
Hu et al.\,\cite{HKMS} related the $G$-regularity notion to a class number divisibility.
In addition they showed the following easy result relating $B$-regularity and $G$-regularity.
\begin{thm}  
\label{thm:GB}
An odd prime $p$ is $G$-regular if and only 
if it is $B$-regular and satisfies $\ord_p(4)=(p-1)/2$.  
\end{thm}
Moree and Sgobba \cite[Proposition~4]{MS1} subsequently placed this characterization in the broader setting of $\ell$-Genocchi numbers defined by
$\ell(1-\ell^n)B_n$. They showed that, for a prime $\ell$ and an odd prime $p\ne\ell$, the prime $p$ is $\ell$-Genocchi regular if and only if it is $B$-regular and $\ord_p(\ell^2)=\frac{p-1}{2}$. They also investigated, using Artin-type methods, the distribution in
arithmetic progressions of primes satisfying this order condition. These results motivate the study of the sets $P_1(a,d)$ and their densities.

Theorem~\ref{thm:p1ad} is an easy consequence of the work of Hu et al.~\cite{HKMS}.

\begin{proof}[Proof of Theorem~\ref{thm:p1ad}]
By Corollary \ref{Opequivalence} we have $\mc{O}(p)=1$ if and only if $p\equiv 3\pmod{4}$ and $\ord_p(4)=(p-1)/2$. Put $D=\lcm(d,4)$. If $a\not\equiv 3 \pmod{(d,4)}$, then the congruences $3\pmod{4}$ and $a\pmod{d}$
are incompatible. Hence $P_1(a,d)(x)=0$, which agrees with the definition of $c(d,a)$. Suppose now that
$a\equiv 3 \pmod{(d,4)}$. By the Chinese remainder theorem, there is a unique residue class
$b\pmod{D}$ such that $b\equiv a\pmod{d}$ and $b\equiv 3\pmod{4}$.
Therefore,
\[
P_1(a,d)(x)= \# \Bigl\{p:p\equiv b\pmod{D},\ \ord_p(4)=\frac{p-1}{2}\Bigr\}.
\]
Since $4\mid D$ and $b\equiv 3\pmod{4}$, Proposition~1.12 of
\cite{HKMS} gives, under GRH,
\begin{equation}
\label{finalasymptotic}    
P_1(a,d)(x)=A \frac{R(D,b)}{\varphi(D)}\frac{x}{\log x}+
O_d\Bigl(\frac{x\log\log x}{(\log x)^2}\Bigr).
\end{equation}
Here the correction factor occurring in \cite[Proposition~1.12]{HKMS} is equal to $1$, since
$4\mid D$ and $b\equiv3\pmod{4}$.

We next compare $R(D,b)$ with $R(d,a)$. For every odd prime $p\mid d$, the congruence $b\equiv a\pmod{d}$ gives $p\mid b-1$ if and only if $p\mid a-1$. Moreover, since $b\equiv 3\pmod{4}$, the contribution of $p=2$ to $R(D,b)$ is $1$ ($1/2$ to the first product and $2$ to the second product). It follows that $R(D,b)=R(d,a)$. Finally,
\[
\frac{\varphi(d)}{\varphi(D)}=
\begin{cases}
1/2 &4\nmid d;\\[1mm]
1&4\mid d.
\end{cases}
\]
Hence
\[
\frac{R(D,b)}{\varphi(D)}=\frac{c(d,a)R(d,a)}{\varphi(d)},
\]
where $c(d,a)$ is as defined in the statement of the proposition. Substitution in \eqref{finalasymptotic} then yields \eqref{eq:underGRH}.\qedhere

\renewcommand{\qed}{}
\end{proof}

\section{The distribution of prime values of \texorpdfstring{$O(p)$}{Op}} \label{prime-values sec}
Motivated by a short note of Sun \cite{Sun}, we consider the frequency with which $\mc{O}(p)$ is itself a prime number. Sun studies the index
\begin{equation}
\label{LP}    
  \ell_p=[\F_p^\ast: \langle-1,2\rangle]=\frac{p-1}{2\ord_p(4)}.
\end{equation}
We note that the second equality in \eqref{LP} is explained in Remark \ref{rem:equalindices}.
Sun, under GRH, determines the ratio of the densities corresponding to odd prime values of $\ell_{p}$ to those corresponding to $\ell_{p}>2$ and shows 
it is between $0.462$ and $0.463$. 
More precisely, he showed that the ratio is 
$6AC/(4-7A)$ with 
$$C\coloneq \sum_{\substack{q\geq 3\\ q\ \textrm{prime}}}\frac{q^2-1}{q^2(q^2-q-1)}
=   \sum_{\substack{q\geq 3\\ q\ \textrm{prime}}} \frac{1}{q^2}\Bigl(1+\frac{q}{q^2-q-1} \Bigr).$$
Using the algorithms (or the Pari/Gp functions) described in Section 5  of Belabas-Cohen \cite{BeCo Book}, it is easy to obtain
\begin{equation}
\label{Cvalue}
  C=0.28488937295069592789537460144246417687\dotsc,
\end{equation}
giving the much more precise value $0.4624263335000077349204\dotsc$ for Sun's ratio.

Theorem~\ref{main} allows us to likewise determine, under GRH, an analogous ratio $R$ for $\mc{O}(p)$. It is the ratio of primes $p$ for which
$\mc{O}(p)$ is prime to those primes $p$ for which $\mc{O}(p)>2$.
\begin{prop}
\label{propR}
Assume GRH. Define
\[
R\coloneq\frac{\sum\limits_{q\geq 3;\ q\ \mathrm{prime}}\delta_q}{\sum\limits_{k\geq 3}\delta_k}
\]
We have
\[
  R=\frac{12AC}{7-12A} = 
  \frac{12A}{7-12A}\sum_{\substack{q\geq 3\\ q\ \mathrm{prime}}}\frac{q^2-1}{q^2(q^2-q-1)}
  =0.5088227112380472975588\dotsc.
  \]
\end{prop}
\begin{proof}
For an odd prime $q$, Theorem \ref{main} gives
\[
  \delta_q= A\frac{q^2-1}{q^2(q^2-q-1)},
\]
showing that the sum in the numerator of $R$ equals $AC$.
It remains to determine the denominator of $R$.
Note that $\sum_{k\geq 1}\delta_k$ is merely the density of primes $p$ for which $\ord_p(4)$ is odd, which by 
Proposition \ref{512} is $7/12$. Since $\delta_1=A$ and $\delta_2=0$, we obtain
\[
  \sum_{k\geq3}\delta_k=\frac7{12}-A,
\]
leading to the claimed expression for $R$. 
Using a sufficiently precise numerical value of the Artin constant and \eqref{Cvalue}, the indicated numerical approximation  of $R$ is obtained.
\end{proof}

\begin{rem} As a sanity check we can try to compute $\sum_{k\ge 1}\delta_k$ directly, making use
of the multiplicativity of $E(k)$.

For an odd prime $q$ we have
\[
  1+\sum_{r\geq1}E(q^r) =1+\frac{q^2-1}{q^2-q-1}\sum_{r\geq1}\frac1{q^{2r}}=\frac{q(q-1)}{q(q-1)-1}.
\]
Therefore, we obtain
\[
  \prod_{q\ \textrm{odd}}\Bigl(1+\sum_{r\geq1}E(q^r)\Bigr)=\prod_{q\ \textrm{odd}} \Bigl(1-\frac{1}{q(q-1)}\Bigr)^{-1}
  =\frac{1}{2A}.
\]
For $q=2$, Theorem \ref{main} shows that the contribution with $\nu_2(k)=1$ vanishes, whereas the contribution with $\nu_2(k)\geq2$ is multiplied by $2/3$. Thus, the corresponding local factor is
\[
  1+\frac23\sum_{r\geq2}\frac{3}{4^r}=1+\frac{1}{6}=\frac{7}{6}.
\]
It follows that
\[
  \sum_{k\geq1}\delta_k=A\cdot\frac{1}{2A}\cdot\frac{7}{6}=\frac{7}{12}.
\]
\end{rem}

\begin{proof}[Proof of Theorem~\ref{CAC-prime-values thm}]Hsia, Li and Sun~\cite{HLScyclo} proved that if $4\nmid\ord_p(2)$ and $[\F_p^\times:\langle-1,2\rangle]$ is an odd prime, then the upper bound in \eqref{M(p) ineq} is attained. Since $\mc{O}(p)>0$ implies
\[
4\nmid\ord_p(2) \qquad\text{and}\qquad \mc{O}(p)=[\F_p^\times:\langle-1,2\rangle],
\]
their result, together with Proposition \ref{propR}, yields the theorem.
\end{proof}

\section{The distribution of \texorpdfstring{$\mc{O}(n)$}{O(n)} for integer \texorpdfstring{$n$}{n}: proof of Theorem  \ref{density O(m) is zero thm}}
\label{Onintegerdistribution}

\begin{proof}[Proof of Theorem \ref{density O(m) is zero thm}]
First assume that $e=0$.
By Lemma \ref{lem-even}  we have $\mc{O}(m)=0$ if and only if $\ord_p(4)$ is even for every prime divisor $p$ of $m$. The integers counted by $N_{0}(x)$, together with $1$, form the completely multiplicative set consisting of the integers all of whose prime divisors belong to $P_{0}$. Applying
\cite[Thm.~4]{Moreediv} with $\tau=\frac{5}{12}$ and $g=4$, we obtain
\begin{equation}
\label{N0xestimate}    
N_{0}(x)=C_{0}\frac{x}{(\log x)^{7/12}} 
+O\Bigl(
\frac{x(\log\log x)^5}{(\log x)^{19/12}}\Bigr)
\end{equation}
for some constant $C_{0}>0$. Applying Proposition 4 of Finch, Martin and Sebah \cite{FMS}, we then obtain the stated formula for $C_0$ (the prime counting
condition of the proposition is satisfied by \eqref{eengelukje}).
Using $t=10^{11}$ one can see that $C_0= 0.32235 \dotsc$
(such decimals are confirmed by comparison with the cases 
$t=10^{j}, j=5,6,7,8,9,10$).

Next assume that $e\ge 1$. Let $p$ be a prime such that $p\equiv 3\pmod{4}$, $\ord_p(2)=p-1$ and $\ord_{p^2}(2)=p(p-1)$, the first few being $p=3,11,19,59,67,83,107,131,139,163, \dots$. Then we have $\ind_p(2)=\ind_{p^2}(2)=1$. 
It then follows by elementary number theory that $\ind_{p^j}(2)=1$ for all $j\geq1$ (see Theorem 2 of \cite{IrRo82}). 
Moreover, for $j\geq1$ we have $\ord_{p^j}(2)=(p-1)p^{j-1} \equiv 2\cdot 3^{j-1}\equiv 2\pmod{4}$.
Hence by Theorem \ref{thm-FLS} we have
\[ \mc{O}(p^e) = \sum_{i=1}^e \ind_2(p^i) = e. \]
In fact, there are infinitely many $n\geq1$ such that $\mc{O}(n)=e$: it suffices to take $n=p^ep_1^{e_1}\dotsm p_s^{e_s}$, where $p$ is as above, $p_1,\dotsc,p_s$ are prime numbers such that $\ord_{p_i}(4)$ is even for all $i$. Then by Lemma \ref{lem-even} we have $\mc{O}(n)=\mc{O}(p^{e})=e$.
Consequently $N_e(x)\ge N_0(x/p^e)$, 
which on invoking the first part of the theorem leads to the claimed lower bound on $N_e(x)$. 
\end{proof}
\begin{rem}
Using the supplementary law of quadratic reciprocity, see \eqref{QR}, it is easy to see that the primes $p=3,11,19,\dotsc$ listed above actually have to satisfy $p\equiv 3\pmod{8}$.
\end{rem}

\section{Proof of Theorem \ref{tight equi code thm}} \label{pf of thm 1 sec}
\label{tightequicodeproof}
\begin{proof}
We first consider prime lengths. By Proposition~\ref{tight equi-diff prop}, if $p\neq 3$ is an
odd prime, then a tight equi-difference conflict-avoiding code of length $p$ exists if and only if $\mc{O}(p)=0$. The prime $p=3$ also gives a tight equi-difference code, since $3\mid p$ and $\mc{O}(3)=1$. 
Thus by Corollary \ref{no of odd cycle cor} we have
$P(x)=1+\#\{p\leq x : 2\mid\ord_p(4)\}$  for $x\ge 3$. 
It then follows by \eqref{P0counting} that
\[
P(x)=\frac{5}{12}\frac{x}{\log x}+O\Bigl(\frac{x(\log\log x)^4}{(\log x)^{3}} \Bigr).
\]
We next consider odd integer lengths. By Proposition \ref{tight equi-diff prop} we infer that
$N(x)=N_{0}(x)+N_{0}(x/3)$,
with $N_0(x)$ defined as in Theorem \ref{density O(m) is zero thm}.
Invoking the estimate \eqref{N0xestimate}, we obtain
\[
N(x)=\frac{4C_{0}}{3}\frac{x}{(\log x)^{7/12}}+O\Bigl(\frac{x(\log\log x)^5}{(\log x)^{19/12}}\Bigr),
\]
with $C_0= 0.32235 \dotsc$
Thus $C=4C_{0}/3=0.4298\dotsc$, completing the proof.\end{proof}

\noindent \textbf{AI Disclosure}. The contributions of P.M., S.P,  and P.S. to this article were made without the use of AI, K. B. used ChatGPT (OpenAI) for language editing and to improve clarity and presentation of the text. \\

\noindent \textbf{Acknowledgment}. 
This paper was initiated by the second author.
The first and third authors joined the project during their stay at the Max Planck Institute for Mathematics (from March to May 2025).  They all thank the staff of the institute for the hospitality and good working conditions. The first author acknowledges financial support from the IIT Goa Start-up Grant (Grant No.~2026/SG/KB/063).
We thank  J-P. Allouche and P. Sol\'e for helpful feedback. A.~Languasco, starting in a later stage in the project, very kindly provided us with
continuous feedback. Furthermore, he numerically approximated the constant $C_0$ in 
Theorem \ref{density O(m) is zero thm}.\\

\end{document}